\documentclass[12pt,reqno]{article}
\usepackage{enumerate}
\usepackage{fullpage}
\usepackage{amsmath, amssymb, amsthm, url, graphicx,rotating,tikz,tkz-graph,relsize}

\newtheoremstyle{plainsl}%
        {\topsep}
        {\topsep}
        {\slshape} 
        {}
        {\normalfont\bfseries}
        {.}
        { }
        {}

\theoremstyle{plainsl}
\newtheorem{theorem}{Theorem}[section]
\newtheorem{proposition}[theorem]{Proposition}
\newtheorem{lemma}[theorem]{Lemma}
\newtheorem{corollary}[theorem]{Corollary}
\newtheorem{definition}[theorem]{Definition}
\newtheorem{example}[theorem]{Example}
\newtheorem{problem}[theorem]{Problem}
\newtheorem{conjecture}[theorem]{Conjecture}

\newcommand\cref[1]{Corollary~\ref{cor:#1}}

\newcommand\sqr[2]{{\vbox{\hrule height.#2pt
    \hbox{\vrule width.#2pt height#1pt \kern#1pt
        \vrule width.#2pt}\hrule height.#2pt}}}
\renewcommand\qed{%
        \ifmmode\eqno\sqr53
        \else\nolinebreak\ \hfill\sqr53\medbreak\fi}

\numberwithin{equation}{section}

\newcommand{\re}{{\mathbb R}}
\newcommand{\cx}{{\mathbb C}}
\newcommand{\ints}{{\mathbb Z}}

\newcommand{\BMA}{{\mathbb A}}

\newcommand{\cD}{\mathcal{D}}

\newcommand{\cE}{{\mathcal E}}  

\newcommand{\hollow}{\mathsf{hollow}}

\newcommand{\cO}{{\mathcal O}}
\newcommand{\cR}{{\mathcal R}}

\newcommand{\sSR}{{\mathsf S}{\mathsf R}}
\newcommand{\scafs}{{\mathsf{s}}}
\newcommand{\scaft}{{\mathsf{t}}}
\newcommand{\cT}{{\mathcal T}}
\newcommand{\Ter}{{\mathbb T}}
\newcommand{\bv}{{\mathbf v}}
\newcommand{\X}{{\mathcal X}}

\newcommand{\one}{{\mathbf{1}}}
\newcommand{\zero}{{\mathbf{0}}}

\newcommand{\SUM}{\mathsf{SUM}}
\newcommand{\Mat}{\mathsf{Mat}}
\newcommand{\sS}{\mathsf{S}}
\newcommand{\tr}{\mathsf{tr}}

\newcommand{\bW}{\mathbf{W}}
\newcommand{\sbW}{\textbf{\fontfamily{cmss}\selectfont W}}

\DeclareMathOperator\Hom{Hom}

\DeclareMathOperator\Sym{Sym}

\DeclareMathOperator\spn{span}

\begin{document}
\thispagestyle{empty}
\setcounter{page}{1}
\title{Scaffolds: A graph-based system for computations in Bose-Mesner algebras}
\author{
William J.~Martin \\
Department of Mathematical Sciences \\
Worcester Polytechnic Institute \\
Worcester, MA USA \\
{\tt  martin@wpi.edu}}

\date{\today} 
\maketitle

\medskip

\begin{abstract}
Let $X$ be a finite set and let $\Mat_X(\cx)$ denote the algebra of matrices with rows and columns indexed by $X$ and entries from the complex numbers acting on $\cx^X$ with standard basis $\{ \hat{x} \mid x\in X\}$.  
For a digraph $G=(V(G),E(G))$, function $R:[m] \rightarrow V(G)$ with $r_j := R(j)$,  and a function $w$ 
from the arcs of $G$ to $\Mat_X(\cx)$, we define the ``scaffold'' $\sS(G,R;w)$ as the sum over all functions $\varphi$ from $V(G)$ to $X$ of the  $m$-fold tensors $\widehat{\varphi(r_1)} \otimes \widehat{\varphi(r_2)} \otimes \cdots 
\otimes \widehat{\varphi(r_m)}$ scaled by  the product of the entries $w(e)_{\varphi(a),\varphi(b)}$ over all arcs $e=(a,b)$ of $G$. If $\Gamma$ is a digraph with adjacency matrix $A$, $m=0$ and $w(e)=A$ for all $e$, this simply counts all digraph homomorphisms from $G$ to $\Gamma$. 
If instead, $\Gamma$ is a distance-regular graph with $i^{\rm th}$ distance matrix $A_i$ and $G$ is a star  
graph $K_{1,m}$ with central node $a$ and degree one nodes $r_1,\ldots,r_m$, the  scaffold 
$\sS(G,R;w)$ encodes generalized intersection numbers 
$\left[ \begin{array}{ccc} a_1 & \cdots & a_k \\ i_1 & \cdots & i_k\end{array} \right]$ 
where $w(e_h) = A_{i_h}$ for $e_h=(a,r_h)$. Scaffolds also arise in the the theory of link invariants and spin models:
the partition function of a link diagram is encoded as a scaffold of order zero and invariance under the  three
Reidemeister moves are encoded as identities among scaffolds of order two and three. 

These diagrams were introduced in the late 1980s by Arnold Neumaier to aid in computations in the
theory of distance-regular graphs.
Various authors have used scaffolds to systematize complex calculations of parameters in association 
schemes. Prioritizing the tutorial value of the paper, we briefly revisit results of Dickie, Suzuki, and 
Terwilliger using this diagrammatic formalism.
Commonly employed transformations are presented as a basic system of ``moves'' on these diagrams 
that preserve their value and  certain natural actions of the Bose-Mesner algebra on nodes and edges 
of a scaffold given by Terwilliger and Jaeger are reviewed. Many of the results presented here are not new; 
our goal is to  collect and present, in a uniform fashion, Neumaier's original idea extended to tensors and
its use by various authors. Sometimes the term ``star-triangle 
diagram'' appears for what we, in this paper,  call ``scaffolds''.  

When one fixes the diagram $G$ and root nodes $R$ but allows the edge weights to vary over matrices in
a given coherent algebra $\BMA$, the vector space  $\sbW((G,R); \BMA)$ spanned by all resulting tensors 
seems worthy of study. We show that $\sbW((H,R'); \BMA)$ is contained in $\sbW((G,R); \BMA)$ when 
$(H,R')$ is a rooted minor of $(G,R)$ and examine several important spaces of this form with three root  
nodes in connection with the Terwilliger algebra.

It is not surprising that some scaffold identities are more intuitive than others: for 
instance, much more is known about distance-regular graphs than is known about $Q$-polynomial 
(``cometric'') association schemes.  Connecting duality in the theory of association schemes to duality
of circular planar graphs, we present a conjecture dealing with dual pairs of planar 
scaffolds which points to a tool for the generation of new identities.
\end{abstract}

\tableofcontents
%
\section{Introduction}
\label{Sec:intro}

Without definition, we begin with some simple examples of the (edge-labeled, rooted) diagrams considered 
in this  paper. Once we 
arrive at a formal definition, the reader may check back to verify the calculations here as exercises. 
For example, we will see how, in a precise way, 
\begin{tikzpicture}[baseline=(A1),black,node distance=0.5cm,
solidvert/.style={draw, circle,  fill=red, inner sep=2.5pt}]
\node[solidvert] (A1) at (0,0) {};
   \path[->] (A1)   [thick, loop] edge node  [right, pos=0.2] {$\scriptstyle{M}$} (A1);
\end{tikzpicture}
 denotes the diagonal of matrix $M$ while 
\begin{tikzpicture}[baseline=(A1),black,node distance=0.5cm,
solidvert/.style={draw, circle,  fill=red, inner sep=2.5pt},
hollowvert/.style={  draw,  circle,  fill=white,  inner sep=2.5pt}]
\node[hollowvert] (A1) at (0,0) {};
   \path[->] (A1)   [thick, loop] edge node  [right, pos=0.2] {$\scriptstyle{M}$} (A1);
\end{tikzpicture}
 is its trace. We
will see matrix $N$, as a second-order tensor, represented as 
\begin{tikzpicture}[baseline=(A1),black,node distance=0.5cm,
solidvert/.style={draw, circle,  fill=red, inner sep=2.5pt}]
\node[solidvert] (A1) at (0,0) {};
\node[solidvert] (A2) at (1,0) {};
   \path[->] (A1)   [thick] edge node  [above] {$\scriptstyle{N}$} (A2);
\end{tikzpicture}
 and the sum of 
its entries as 
\begin{tikzpicture}[baseline=(A1),black,node distance=0.5cm,
solidvert/.style={draw, circle,  fill=red, inner sep=2.5pt},
hollowvert/.style={  draw,  circle,  fill=white,  inner sep=2.5pt}]
\node[hollowvert] (A1) at (0,0) {};
\node[hollowvert] (A2) at (1,0) {};
   \path[->] (A1)   [thick] edge node  [above] {$\scriptstyle{N}$} (A2);
\end{tikzpicture}
. The ordinary matrix product of $M$ and $N$ is encoded as 
a series reduction \input{arxivdiagrams/series2.tex} and entrywise multiplication is a parallel reduction
\input{arxivdiagrams/parallel2.tex}. If $\Gamma$ is a graph with adjacency matrix $A$, then \input{arxivdiagrams/4cycle.tex}
counts all homomorphisms from the cycle of length four into graph $\Gamma$ while 
\input{arxivdiagrams/induced4cycle.tex} counts, for each vertex $v$ of $\Gamma$, 
twice the number of induced cycles of length four passing through $v$, where $A'=J-I-A$ 
is the adjacency matrix of the complement of $\Gamma$.

Originating from unpublished notes of Neumaier (ca.\ 1989) and tensor calculations of Terwilliger \cite{terPQ},
diagrams of this sort seem to have been shared informally in the community for several decades now. 
Equivalent algebraic formulations appear in the work of Dickie \cite{dickie} and Suzuki \cite{suztwoq,suzimprim}.
An important special case arises in the state models for link invariants as seen, for example, in
Jaeger \cite{jaeger}. The primary goal of this paper is to present the diagrammatic formalism 
as a rigorous alternative to the more cumbersome algebraic expressions that these diagrams represent. Write
$[m]=\{1,\ldots,m\}$. For a
digraph $G=(V(G),E(G))$, $R: [m] \rightarrow V(G)$ with $r_j:=R(j)$, and a function $w$ mapping the edges of $G$ to matrices with rows and
columns indexed by a finite set $X$, we study the tensor
\begin{equation}
\label{Eq:scaffshort}
\sS(G,R;w ) = \sum_{ \varphi:V(G) \rightarrow X} \quad \left( \prod_{\substack{ e\in E(G)  \\ e=(a,b) }} w(e)_{\varphi(a),\varphi(b)} \right) \widehat{\varphi(r_1)} \otimes \widehat{\varphi(r_2)} \otimes \cdots  \otimes \widehat{\varphi(r_m)}
\end{equation}
and a slight generalization thereof. Note here that $w(e)$ is a matrix and $w(e)_{\varphi(a),\varphi(b)} $ is simply
the entry of that matrix which appears in row $\varphi(a)$, column $\varphi(b)$.

As this project neared its conclusion, a  pair of manuscripts by Penjic and Neumaier  appeared  
\cite{neupen1,neupen2}.  To the author's knowledge, this is the first published record of 
Neumaier's diagrammatic notation. In these papers, focus is placed
on zeroth order scaffolds (scalars) where the edge weights are distance matrices in a metric association
scheme and many classical inequalities for distance-regular graphs are recovered, in a unified fashion, 
through the clever manipulation of inequalities, with a different 
scaling from what we use here.

\subsection{Basic Notation}
\label{Subsec:basics}

Let $X$ be a finite set and let $\Mat_X(\cx)$ denote the vector space of matrices with rows and columns 
indexed by $X$ and entries from the complex numbers. Take $V=\cx^X$ with standard basis 
$\{ \hat{x} \mid x\in X\}$, equipped with the corresponding positive definite Hermitean inner product 
$\langle v,w\rangle = v^\dagger w$ (where $\cdot^\dagger$ denotes conjugate transpose) 
satisfying $\langle \hat{x},\hat{y} \rangle = \delta_{x,y}$ for $x,y\in X$; 
this allows us to identify $V$ with its dual space $V^\dagger$ of linear functionals.  The objects
of study belong to tensor products of this space of the form 
$$  V^{\otimes m} = \underbrace{V \otimes V \otimes \cdots \otimes V}_{m} $$
with standard basis consisting of simple tensors of the form 
$\hat{x}_1 \otimes \hat{x}_2 \otimes \cdots \otimes \hat{x}_m$ where
$x_1,x_2,\ldots, x_m \in X$. When $m=2$, we identify $V \otimes V$ with $\Mat_X(\cx)$ via  
$$M=[M_{xy}]_{x,y} \  \longleftrightarrow \ \sum_{x,y \in X} M_{xy} \ \hat{x} \otimes \hat{y}. $$

Clearly $\Mat_X(\cx)$ forms an algebra 
both under matrix multiplication and under entrywise (Hadamard, or Schur) multiplication, which we denote by 
$\circ$, and contains the identities, $I$ and $J$ respectively, for these two multiplications.  
Any subring 
$\BMA$ of $\Mat_X(\cx)$ acts by multiplication from the left on $V=\cx^X$ and 
$\BMA^{\otimes m}$ acts on $V^{\otimes m}$ component by component:
$$(A_1 \otimes   \cdots   \otimes A_m) ( \bv_1 \otimes \cdots   \otimes\bv_m ) = A_1 \bv_1 \otimes \cdots   
\otimes A_m \bv_m .$$

Our primary example for the vector space $\BMA$ will be  the Bose-Mesner algebra 
$\BMA = \spn_\cx \{ A_0, \ldots, A_d\}$ of a commutative $d$-class association scheme or, relaxing the commutativity condition, a coherent algebra. But the tools here 
clearly extend to other settings, for example where 
$\BMA=\langle A\rangle$  is  the adjacency algebra of a finite simple graph $\Gamma$ with adjacency matrix $A$.

\subsection{Scaffolds (or ``Star-Triangle Diagrams'')}
\label{Subsec:scaffolds}

Suppose we are given
\begin{itemize}
\item A finite (di)graph $G=(V(G),E(G))$ possibly with loops and/or multiple edges, 
the \emph{diagram} of the scaffold\footnote{Note that we write $e=(a,b)$ to indicate that edge $e$ has tail $a$ and
head $b$; this is a slight abuse of notation in the presence of parallel edges.};
\item An ordered subset $R=\{r_1,\ldots,r_m\} \subseteq V(G)$  of ``root'' nodes; or, more generally, a
function $R:[m] \rightarrow V(G)$ with $r_j := R(j)$, In the language of \cite[p39]{lovasz}, $G$, together 
with $R$, is a ``$k$-multilabeled graph'', but we will call $(G,R)$ a \emph{rooted diagram};
\item A finite set $X$ and a map from edges of $G$ to matrices in $\Mat_X(\cx)$: \\
$ w: E(G) \rightarrow \Mat_X(\cx) $ (\emph{edge weights});
\item a subset $F \subseteq V(G)$ of  \emph{fixed nodes} and a fixed function $\varphi_0:F\rightarrow X$
\end{itemize}
The {\em (general) scaffold} $\sS(G, R;w;F,\varphi_0)$ is 
defined as the quantity
\begin{equation}
\label{Eq:scaffdefn}
\sS(G,R; w; F,\varphi_0 ) = \sum_{\substack{ \varphi:V(G) \rightarrow X\\
(\forall a\in F)(\varphi(a)=\varphi_0(a) )}} \quad \left( \prod_{\substack{ e\in E(G)  \\ e=(a,b) }} w(e)_{\varphi(a),\varphi(b)} 
\right) 
 \widehat{\varphi(r_1)} \otimes  \widehat{\varphi(r_2)} \otimes \cdots \otimes  \widehat{\varphi(r_m)}.
\end{equation}
Each function $\varphi:V(G) \rightarrow X$ whose restriction to $F$ is $\varphi_0$ is called a \emph{state} of the
scaffold. Observe that each state itself yields a scaffold with one summand by taking $F=X$ and $\varphi_0 = \varphi$,
namely 
$$ w(\varphi) \  \widehat{\varphi(r_1)} \otimes  \widehat{\varphi(r_2)} \otimes \cdots \otimes  \widehat{\varphi(r_m)}$$
where the  \emph{weight} of $\varphi$ is defined as
$w(\varphi) :=  \displaystyle{\prod_{\substack{ e\in E(G)  \\ e=(a,b) }}} w(e)_{\varphi(a),\varphi(b)}$.

Reversing an arc $e=(a,b)$ in diagram $G$ is equivalent to replacing $w(e)$ by its transpose. In the case where
all edge weights are symmetric matrices, we may treat $G$ as an undirected graph.

The scaffold $\sS(G, R ;w; F,\varphi_0 )$ is an element of $V^{\otimes m}$, so we say 
$\sS(G, R ;w; F,\varphi_0 )$ is a scaffold of {\em order} $m$. 
Scaffolds with $m=0$ are simply complex numbers.
In our discussion, $X$ will always denote the vertex set of some graph or association scheme.
To distinguish $V(G)$ from $X$, we will refer to elements of $V(G)$ as \emph{nodes}.
As the examples above and below show, a scaffold on a small number of nodes can often be concisely
encoded pictorially as an edge-labeled diagram once a convention is established for the ordering 
$r_1,\ldots,r_m$ of root nodes. (Unless  noted explicitly, we assume the $r_j$ are 
distinct; i.e., that the labeling $R$ is injective.) 
Viewed as elements of $\cx^{X^m}$, we may take linear combinations of such scaffolds 
as needed.  We may multiply scaffolds as well: the vector space  
$$ T(X) = \bigoplus_{m=0}^\infty  \cx^{X^m} $$
is viewed as the algebra of tensors by linear extension of the canonical isomorphism 
$$\cx^{X^m} \otimes \cx^{X^p} \rightarrow \cx^{X^{m+p}} . $$ 
For example, if the diagram $G$ is not a connected graph and $G$ can be expressed as a disjoint union of two (not 
necessarily connected) graphs  $H_1$ and $H_2$ where $R = \{r_1,\ldots,r_m\}$ is ordered so that
$R_1 = \{r_1,\ldots,r_\ell\} \subseteq V(H_1)$ and $R_2 = \{r_{\ell+1},\ldots,r_m\} \subseteq V(H_2)$,
then
$$\sS(G, R ;w) = \sS(H_1 , R_1 ;w_1) \otimes \sS(H_2 , R_2; w_2) $$
where $w_j(e)=w(e)$ for edge edge $e$ of $G$.

Fixing the rooted diagram $(G,R)$, the set of all scaffolds $\sS(G, R; w)$ as $w$ varies over all functions from 
$E(G)$ into $\BMA$, we obtain (cf.~\cite{jaeger})  a multilinear map from $\BMA^{\otimes E(G)}$ to $\cx^{X^m}$. 
In Section \ref{Sec:vectorspaces}, we study the images of such maps.

In this paper, with few exceptions,  the set $F$ of fixed nodes is empty. In this case, the expression 
takes the simpler form given in (\ref{Eq:scaffshort}) which we repeat here:

$$\sS(G,  R ;w ) = \sum_{ \varphi:V(G) \rightarrow X} \quad \left( \prod_{\substack{ e\in E(G)  \\ e=(a,b) }} 
w(e)_{\varphi(a),\varphi(b)} \right) \widehat{\varphi(r_1)} \otimes \widehat{\varphi(r_2)} \otimes \cdots  \otimes \widehat{\varphi(r_m)}.  $$
Let us call tensors of this form \emph{symmetric scaffolds} when the distinction is necessary.\footnote{One might consider a more general tensor of this sort as follows. Let $\hat{X}$ be some basis for $\cx^X$ --- one might choose the standard basis as we have done, choose an eigenbasis, or some other basis. Rather than sum over all functions from $V(G)$ to $X$, we may instead sum over all functions from $V(G)$ to  $\hat{X}$ and define
$\sS(G, R;w) = \sum_{\varphi:V(G) \rightarrow \hat{X} } \quad \left( \prod_{\substack{ e\in E(G)  \\ e=(a,b) }} \varphi(a)^\dagger w(e) \varphi(b)  \right) \varphi(r_1) \otimes \cdots \otimes  \varphi(r_m)$.}

In the case where $R = \emptyset$, scaffolds evaluate to scalars and 
we recover Jaeger's definition of a \emph{partition function} in 
\cite[Section 1]{jaeger} and the original counting diagrams of Neumaier (cf.  \cite{neupen1,neupen2}).

\begin{example}
Suppose $A$ is the adjacency matrix of the following graph $\Gamma$ on vertex set $X=\{u,v,w,x,y\}$:
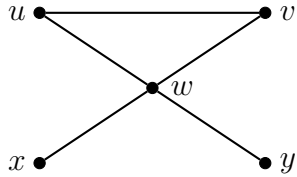
\begin{center}
\input{arxivdiagrams/irregexamp.tex}
\end{center}
and $A'=J-I-A$. Then we have the following third order scaffold with 
$$G=(V(G),E(G)), \qquad V(G) = \{1,2,3,4\}, \qquad  R:[3] \rightarrow V(G), \ r_j=j$$
$$  E(G)=\{(1,2),(2,3),(1,4),(3,4) \}, \qquad w(3,4)=A'$$
and all other edge weights $w(e)=A$:
\begin{center}
\input{arxivdiagrams/irregscaff.tex}
\end{center}
For general scaffolds (when the set $F$ of fixed nodes is non-empty), we label each $a\in F$ by $\varphi_0(a)$ as
indicated in the following examples:
\begin{center}
\input{arxivdiagrams/irregGeneral.tex} \ \ $\Box$
\end{center}
\end{example}


\subsection{Rules for scaffold manipulation}
\label{Subsec:rules}

We start with simple examples in order to illustrate the notation and exhibit elementary 
properties. All edge weights are assumed to be matrices in
$\Mat_X(\cx)$. Here's how we extract the diagonal of a matrix and sum along the diagonal to obtain the trace:

\begin{center}
\input{arxivdiagrams/SimplestScaffold.tex} 
\end{center}

\noindent Note that, when edge weights are taken from  a  Bose-Mesner algebra, where all matrices 
have constant diagonal, loops can always be removed once this scalar is accounted for.

Next, we represent a matrix $N$ as a  second order tensor,  the vector $N \one$ as a first order 
tensor, and the sum of its entries $\one^\top N \one$ as a scalar, where $\one$ denotes the vector in $\cx^X$ having all entries equal to one:

\begin{center}
\input{arxivdiagrams/SingleEdgeScaffolds.tex} 
\end{center}

\noindent Henceforth, we identify the matrix $N\in \Mat_X(\cx)$ with the second order tensor 
$\displaystyle{\sum_{x,y\in X}} N_{xy} \hat{x}\otimes \hat{y}$.

A special case occurs when $G$ is an edgeless graph; if $|V(G)|=n$ and $|R|=m$, then the corresponding
scaffold is $ |X|^{n-m} \sum_{x_1,\ldots,x_m \in X} \hat{x}_1 \otimes \hat{x}_2 \otimes \cdots \otimes \hat{x}_m$.
For example, when $n=m=2$, we obtain the all ones matrix $J$.

The fundamental scaffold identities below show that matrix product and Schur/entrywise product correspond,
respectively, to series and parallel reductions on diagrams:

\input{arxivdiagrams/twoedge.tex}    

Note that the first identity is only guaranteed to hold when the middle (hollow) node on the top left is incident 
only to the two edges shown. We present these identities, denoted in Appendix A as $\sSR{1}$ and $\sSR{1'}$, 
in the following lemma.

\begin{lemma}
\label{Lem:seriespar}
Let $G=(V(G),E(G))$ with $e_1,e_2\in E(G)$ where $e_1=(a_1,b_1)$ and $e_2=(a_2,b_2)$. Let 
$$E'= \left( E(G) \setminus \{ e_1, e_2 \} \right) \cup \{ e'=(a_1,b_2) \} $$
and $G'=(V(G),E')$.
\begin{itemize}
\item[(i)] If $e_1$ and $e_2$ are in series, i.e., $b_1=a_2$ and no other edge is incident to $b_1$, then
$$ \sS(G, R ;w) = \sS(G', R; w' )$$
whenever $b_1 \not\in R$, where $w'(e') = w(e_1) w(e_2)$ and $w'(e)=w(e)$ otherwise.
\item[(ii)] If $e_1$ and $e_2$ are in parallel, i.e., $a_1=a_2$ and $b_1=b_2$, then
$$ \sS(G, R ; w) = \sS(G' ,  R ; w' )$$
where $w'(e') = w(e_1) \circ w(e_2)$ and $w'(e)=w(e)$ otherwise. \ $\Box$
\end{itemize}
\end{lemma}

Another basic lemma:
\begin{lemma}
\label{Lem:trivial}
Let $G=(V(G),E(G))$ with $e=(a,b) \in E(G)$.
\begin{itemize}
\item[(i)] Let $G'$ be obtained from $G$ by deleting edge $e'$. If $w(e')=J$, 
then $\sS(G , R ; w)= \sS(G' , R ; w')$ where
$w'(e)=w(e)$ for all $e \in E(G) \setminus \{e'\}$.
\item[(ii)] Let $G'$ be obtained from $G$ by contracting edge $e'$:  for $f(a)=f(b)=a'$ 
and $f(c)=c$ otherwise,
$$V(G')= \left( V(G) \setminus \{a,b\} \right) \cup \{a'\}, \quad E(G') = \left\{ (f(u),f(v)) 
\mid (u,v) \in E(G) \right\}$$
if $w(e')=I$ and either $a\not\in R$ or $b\not\in R$, then $\sS(G , R ; w)= 
\sS(G' , R' ; w')$ where $w'( f(u),f(v) )=w(u,v)$ for all $e =(u,v) \in E(G) \setminus 
\{e'\}$, $R'=R$ if $a,b\not\in R$ and $R'=R\cup \{a'\}$ otherwise. \ $\Box$
\end{itemize}
\end{lemma}

This second lemma provides us with two more basic moves, $\sSR{0}$ and $\sSR{0'}$ in our 
notation, that preserve scaffolds:
\begin{itemize}
\item  Split a node in two, introducing a hollow node, mapping the new edge to $I$ \\
\input{arxivdiagrams/splitnode.tex} \\
\noindent Note that the edges incident to the original node may be distributed among the two nodes on the right
in any fashion.
\item  Insert an edge anywhere, between two existing nodes, mapping the new edge to $J$ \\
\input{arxivdiagrams/insertJ1.tex}  \hspace{.4in}
or  \hspace{.4in}
\input{arxivdiagrams/insertJ2.tex}  \hspace{.2in} . 
\end{itemize}
In the case where edge weights belong to a Bose-Mesner algebra, these steps are useful in conjunction 
with the linear expansions $I = \sum_j E_j$ and $J= \sum_i A_i$, respectively, which allow us to expand one scaffold 
as a linear combination of closely related scaffolds.

It is a simple exercise to show that \ \  \input{arxivdiagrams/rrw.tex} \ \  is equal to $M \circ (JN^\top)$.
More generally, if our rooted diagram contains a hollow node of degree one and the incident edge weight has constant 
row and column sum, we may simplify the scaffold by 
deleting this node and scaling the resulting tensor by that constant (Rule $\sSR{9}$).

\begin{lemma}
\label{Lem:degreeonehollow}
Let $G=(V(G),E(G))$ be a digraph, $w:E(G) \rightarrow \Mat_X(\cx)$, 
$R:[m]\rightarrow V(G)$ an injection, $a_0 \in V(G)$ not in the image of $R$ and 
incident to just one edge in $E(G)$ (say $e=(a_0,b_0)$ or $e=(b_0,a_0)$). Let $M=w(e)$. Denote by $G'$ the graph
obtained from $G$ by deletion of $a_0$ and $e$; denote by $w'$ the restriction of $w$ to edges other than $e$.
\begin{itemize}
\item[(i)]  if $e=(b_0,a_0)$ and $M$ has constant row sum $\alpha$, then
$\sS( G , R ; w) = \alpha \sS( G' , R ; w')$;
\item[(ii)]  if $e=(a_0,b_0)$ and $M$ has constant column sum $\alpha$, then
$\sS( G , R ; w) = \alpha \sS( G' , R ; w')$.  $\Box$
\end{itemize}
\end{lemma}

General scaffolds become useful when one's investigation differentiates individual vertices 
of an object or when one wishes to expand an expression as a sum over vertices. If 
$$\scaft = \sS( G , R ; w)$$ 
is a symmetric scaffold with edge weights in $\Mat_X(\cx)$ and $r\in R$, define, for $x\in X$,
$$\scaft_x = \sS( G , R ; w; \{r\}, \varphi_x)$$
where $\varphi_x(r)=x$. Then we have $\scaft = \displaystyle{\sum_{x\in X}} \  \scaft_x$:
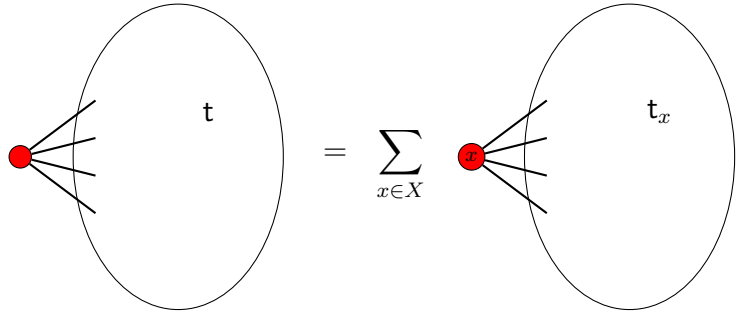
\begin{center}
\input{arxivdiagrams/splitoververtices.tex}
\end{center}

Using the notion of a general scaffold, we obtain a straightforward proof that our manipulations of 
subdiagrams are valid operations on overall diagrams. Let $G=(V(G),E(G))$ and $H=(V(H),E(H))$ be
digraphs with $V(G) \cap V(H) = \{u_1,\ldots, u_\ell\}$, or with a bijection $\xi$ pairing $\ell$ nodes $u_i$ of $G$
to $\ell$ nodes $r_i$ of $H$ (e.g, $\xi$ is the identity function in the simple case). Writing $\xi(u_i)=r_i$, 
define $G +_\xi H$ to be the digraph with vertex set
$$ V(G +_\xi H) = \left(  V(G)  \setminus \{ u_1,\ldots u_\ell\} \right) \cup V(H) $$
and edges
$$ E(H) \cup  \{ (\xi(a),\xi(b)) \mid (a,b) \in E(G) \} $$
where we extend $\xi$  to $V(G)$ defining $\xi(c)=c$ for $c\not\in \{u_1,\ldots,u_\ell\}$. If $G$ and $H$ are 
rooted diagrams, we define $\hat{R}$ to be the union of their respective sets of root nodes, replacing $u_i$
by $\xi(u_i)$ whenever $u_i$ is a root node.

\begin{proposition}
\label{Prop:glue}
 If  $\scaft_1$ and $\scaft_2$ are symmetric scaffolds on the same set $R$ of roots such 
that $\scaft_1 = \scaft_2$ and $R' =\{r_1,\ldots,r_\ell\} \subseteq R$, then for any scaffold $\scafs$ and any 
root nodes $u_1,\ldots,u_\ell$ in the rooted diagram of $\scafs$, we have $\scafs +_\xi \scaft_1 = \scafs +_\xi \scaft_2$
where $\xi(u_i) = r_i$.
\end{proposition}

\begin{proof}
Write $\scafs = \sS( G, P; w)$, $\scaft_j = \sS( H_j, R; w_j )$ for $j=1,2$.
For any $x_1,\ldots,x_\ell \in X$, consider $\phi_0: r_i \mapsto x_i$ and $F=\{u_1,\ldots,u_\ell\}$. It is easy to
see that 
$$\sS(  H_1, R; w_1;  R', \phi_0) = \sS(  H_2, R; w_2;  R', \phi_0) $$
and $\sS(  G, P; w ;  \xi^{-1}(R'), \phi_0 \circ \xi )$  is obviously equal to itself. Writing $\hat{R} = \{r_1,\ldots,r_m\}$, 
we have
$ \displaystyle{ \sum_{\substack{ \varphi: V(G +_\xi H_1) \rightarrow X\\
(\forall a\in R')(\varphi(a)=\varphi_0(a) )}} \quad \left( \prod_{\substack{ e\in E(G +_\xi H_1)  \\ e=(a,b) }} 
w(e)_{\varphi(a),\varphi(b)}  \right) }
 \widehat{\varphi(r_1)} \otimes  \widehat{\varphi(r_2)} \otimes \cdots \otimes  \widehat{\varphi(r_m)}$
 
\ \ \ $ = \displaystyle{ \sum_{\substack{ \varphi: V(G +_\xi H_2) \rightarrow X\\
(\forall a\in R')(\varphi(a)=\varphi_0(a) )}} \quad \left( \prod_{\substack{ e\in E(G +_\xi H_2)  \\ e=(a,b) }} 
w(e)_{\varphi(a),\varphi(b)}  \right) }
 \widehat{\varphi(r_1)} \otimes  \widehat{\varphi(r_2)} \otimes \cdots \otimes  \widehat{\varphi(r_m)}$
 
\noindent since corresponding coefficients are equal. Now summing the first of these expressions over all possible choices of $x_1,\ldots,x_\ell \in X$ gives us $\scafs +_\xi \scaft_1$ and summing the second over all possible 
choices of $x_1,\ldots,x_\ell \in X$ gives us $\scafs +_\xi \scaft_2$.  $\Box$
\end{proof}

This tool is denoted $\sSR{5}$ in Appendix A. Together with our ability to convert solid nodes to hollow nodes 
while preserving scaffold equality (Lemma \ref{Lem:orderred}, $\sSR{10}$), this allows us to make local 
moves on scaffolds.

%
%
%
\section{Broad utility of diagrams}
\label{Sec:otherstuff}  

Before we delve into the main line of investigation, namely computations in Bose-Mesner algebras, we 
illustrate the broader utility of Neumaier's diagrams by mentioning two other active areas of combinatorics: homomorphism densities and spin models. The point here is that concepts closely related to these 
have been used for quite some time in a variety of areas close to combinatorics.

\subsection{Counting homomorphisms} 
\label{Subsec:homomorphisms}  

Throughout this section, let $\Gamma$ denote a finite simple graph with adjacency matrix $A$. For a graph 
$G$, let $\Hom(G,\Gamma)$ denote the set of all graph homomorphisms $\varphi: G \rightarrow \Gamma$. For 
$u_1,\ldots, u_m \in V(G)$ and $v_1,\ldots, v_m \in V(\Gamma)$, denote by 
$$ \Hom_{u_1,\ldots,u_m}^{v_1,\ldots,v_m}(G,\Gamma)$$
the set of $\varphi \in \Hom(G,\Gamma)$  with $\varphi(u_i)=v_i$. for all $i$. Then with $w(e)=A$ for all
$e\in E(\Gamma)$, we have 
$$\sS(G , \emptyset ; w) = \left|  \Hom(G,\Gamma)\right| ,$$
$$ \sS(G , \{u_1,\ldots,u_m \} ; w) = \sum_{v_1,\ldots,v_m} \left| \Hom_{u_1,\ldots,u_m}^{v_1,\ldots,v_m}(G,\Gamma)\right|
\hat{v}_1 \otimes \cdots \otimes \hat{v}_m $$
and, defining $\varphi_0(u_i)=v_i$ for each $i$,
$$ \sS(G , \emptyset ; w; \{u_1,\ldots,u_m \},\varphi_0) = 
\left| \Hom_{u_1,\ldots,u_m}^{v_1,\ldots,v_m}(G,\Gamma)\right|
~ . $$
For example, $\sS(K_3 , \emptyset ; A ) =$ \input{arxivdiagrams/hollowtriangle.tex} counts triangles in $\Gamma$ while \\
 \hspace*{1.5in} \input{arxivdiagrams/markedtriangle.tex} \vspace{-0.65in} 
 
 $$  \phantom{XXXXX} = \displaystyle{\sum_{(v_1,v_2)\in V\Gamma} }\left| \Gamma(v_1)\cap \Gamma(v_2) \right| \hat{v}_1 \otimes \hat{v}_2 $$
records, for each pair of vertices, their number of common neighbors.

The \emph{homomorphism density} $\mathsf{t}(G,\Gamma)$ of $G$ into $\Gamma$ is defined as the probability
that a function from $V(G)$ to $V\Gamma$, chosen uniformly at random, is a graph homomorphism
$$  \mathsf{t}(G,\Gamma) = \left|  \Hom(G,\Gamma) \right|/ |V\Gamma|^{|V(G)|}. $$
Lov\'{a}sz's  theory of graphons \cite{lovasz} includes a substantial set of tools for computing 
homomorphism densities, with powerful 
applications. We note here that our diagrammatic notation differs from that of Lov\'{a}sz slightly.

By embedding $G$ in a complete graph on $|V(G)|$ nodes, we obtain added flexibility. Edges of the new 
scaffold diagram which are not edges of $G$ may be labelled with $J-I$ to count injections from $G$ into $\Gamma$, or with $J-I-A$ to
count induced subgraphs of $\Gamma$ which are isomorphic to $G$.

\subsection{Partition functions and spin models} 
\label{Subsec:spinmodels}

If $X$ is a finite set of colors (or ``spins''), we may employ a partition function to assign a complex number to each link diagram. The concept of a \emph{spin model} plays a key role here in determining which partition
functions are link invariants.

A spin model \cite[Prop.~1]{jaeger} is a triple $(X,W^+,W^-)$ where $X$ is a finite set and matrices 
$W^+,W^- \in \Mat_X(\cx)$ satisfy, for $D^2=|X|$ and some scalar $\alpha$, 
\begin{itemize}
\item (``Type I relation'' for $W^+$) $W^+$ has  constant diagonal $\alpha$,  constant row sum $D \alpha^{-1}$, and constant column sum $D \alpha^{-1}$;
\item  (``Type I relation'' for $W^-$) $W^-$ has constant diagonal $\alpha^{-1}$,  constant row sum $D \alpha$, and constant column sum $D \alpha$; 
\item  (``Type II relation'') $W^+ \circ (W^-)^\top = J$ while $W^- W^+ = W^+ W^- = |X| I$
\item ( ``Type III'' or``Star-Triangle Relation'') for every $a,b,c\in X$,
$$ \sum_{x \in X} (W^+)_{x,a}  (W^+)_{x,b}  (W^-)_{c,x} = D \  (W^+)_{a,b} (W^-)_{b,c}(W^-)_{c,a} $$
\end{itemize}

In scaffold formalism, these conditions are written
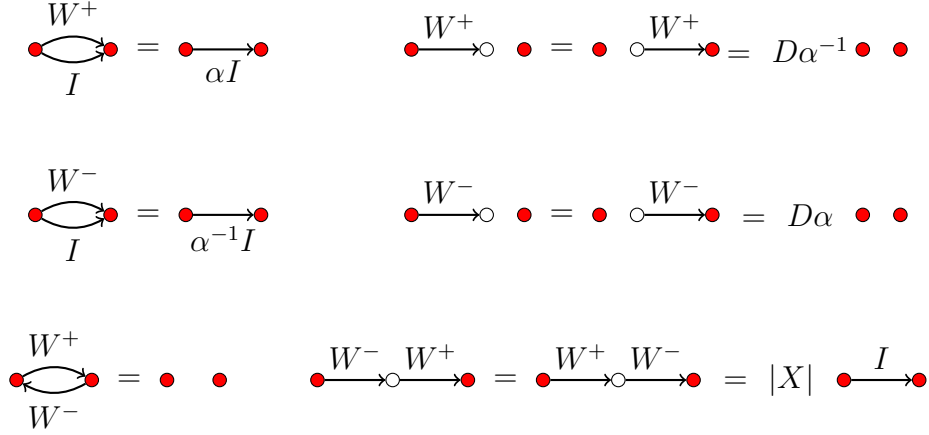
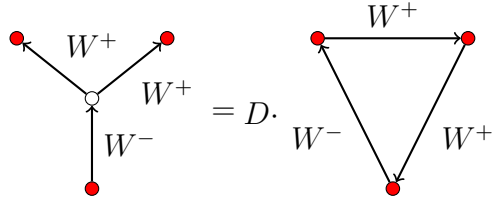
\begin{center}
\input{arxivdiagrams/spin_model.tex}
\end{center}

\begin{example}
The \emph{cyclic spin model} introduced by Goldschmidt and Jones 
is given by $X = \ints_n$ and $W^+_{a,b} = \omega^{(a-b)^2}$ where $\omega$ is a primitive complex  $n^{\rm th}$ root
of unity for $n$ odd. For $n=5$ with $\omega=e^{2\pi i/5}$, one may readily check that 
$$ W^+ = \left[ \begin{array}{ccccc}
       1            &      \omega    &  \bar{\omega}   &  \bar{\omega} &       \omega \\
 \omega        &          1          &       \omega      &  \bar{\omega} &  \bar{\omega}  \\
 \bar{\omega}&     \omega     &           1            &       \omega    &  \bar{\omega}  \\
 \bar{\omega}& \bar{\omega}  &      \omega      &             1         &      \omega  \\
   \omega     & \bar{\omega}   &  \bar{\omega}  &        \omega    &           1  \end{array} \right] \qquad \mbox{and} \qquad
W^- = \left[ \begin{array}{ccccc}
       1            &      \bar{\omega}    &  \omega   &  \omega &       \bar{\omega} \\
 \bar{\omega}        &          1          &       \bar{\omega }     &  \omega &  \omega  \\
 \omega&     \bar{\omega}     &           1            &       \bar{\omega}    &  \omega  \\
 \omega & \omega  &      \bar{\omega}      &             1         &      \bar{\omega}  \\
  \bar{ \omega}     & \omega   &  \omega  &        \bar{\omega}    &           1  \end{array} \right]$$
satisfy the above relations of Type I, II and III.
\end{example}

In \cite{gitlerlopez}, Gitler and L\'{o}pez investigate a generalization of the star-triangle relation exploring connections
between highly regular association schemes and spin models.

%
%
%
\section{Association schemes}  
\label{Sec:schemes}

A {\em (commutative) association scheme} \cite{del,banito,bcn,godsil}  
consists of a finite set $X$ together with a collection
$\cR=\{R_0,\ldots,R_d\}$ of binary relations (the \emph{basis relations}) on $X$ satisfying the following
conditions:
\begin{itemize}
\item[(i)] some relation in $\cR$ is the identity relation on $X$; 
we denote this relation $R_0$;
\item[(ii)] $R_i \cap R_j = \emptyset$ whenever $i\neq j$;
\item[(iii)] $R_0 \cup \cdots \cup R_d = X\times X$;
\item[(iv)] for each $i$, the relation
$R_i^\top = \left\{ (b,a) : (a,b) \in R_i \right\} $ also belongs to $\cR$;
\item[(v)] there  are \emph{intersection numbers} $p_{ij}^k$ ($0\le i,j,k\le d$)
such that, whenever $(a,b)\in R_k$, we have exactly $p_{ij}^k$
elements $c\in X$ for which both $(a,c)\in R_i$ and $(c,b)\in R_j$;
\item[(vi)] for each $i$, $j$ and $k$, $p_{ij}^k = p_{ji}^k$. 
\end{itemize}
An association scheme is \emph{symmetric} if $R_i^\top = R_i$ for all $R_i \in \cR$.

Let $(X,\cR)$ be an association scheme   with basis relations  $\cR=\{R_0,R_1,\ldots, R_d\}$ having adjacency matrices $A_0,A_1,\ldots, A_d$  respectively. It is well known  \cite[Thm.~2.6.1]{bcn} that  the vector space $\BMA$ 
spanned by these $d+1$ matrices  is a Bose-Mesner algebra\footnote{I.e., the vector space is closed under 
transpose, closed under conjugation, closed and commutative under both ordinary and  entrywise 
multiplication, and contains the identities, $I$ and $J$,  respectively, for these two operations.}  
and that  $\{ A_0,A_1,\ldots, A_d\}$ forms a basis of pairwise orthogonal 
idempotents with respect to the entrywise product: $A_i \circ A_j = \delta_{i,j} A_i$. By convention, 
we have $A_0=I$ and the unique basis of orthogonal idempotents with respect to ordinary matrix 
multiplication is denoted by $\{E_0,E_1,\ldots, E_d\}$ with $E_0 = \frac1{|X|} J$. It follows that there exist
structure constants, called \emph{Krein parameters}, $q_{ij}^k$ for which
$$ E_i \circ E_j = \frac{1}{|X|} \sum_{k=0}^d q_{ij}^k E_k $$
for $0\le i,j\le d$. Note that, for each $j \in \{0,\ldots,d\}$, there is some 
$j' \in \{0,\ldots,d\}$ with $E_{j'}=E_j^\top=\bar{E}_j$. 
The two bases $\{A_0,\ldots, A_d\}$ and $\{E_0,\ldots,E_d\}$ for algebra $\BMA$ are related by
the \emph{first and second eigenmatrices} $P$ and $Q$ defined by 
$$ A_i = \sum_{j=0}^d P_{ji} E_j \qquad E_j = \frac{1}{|X|} \sum_{i=0}^d Q_{ij} A_i. $$
These satisfy the \emph{orthogonality relations} \cite[Sec.~2.2]{bcn} $PQ=|X| I$ and $m_j P_{ji} = v_i Q_{ij}$
where $m_j = q_{jj}^0$ is the rank of $E_j$ and $v_i = p_{ii}^0$ is the valency of the graph $(X,R_i)$ corresponding
to the $i^{\rm th}$ basis relation.

Let us first consider scaffolds $\sS(G , R ;w)$ 
where $w(e)\in \{A_0,\ldots, A_d\}$ for every edge $e$ in $G$. Note that we are abusing our conventions in that the set $X$ is not specified: when an identity is given, we mean that it holds true (under the given hypotheses) 
for any association scheme $(X,\cR)$ whose Bose-Mesner $\BMA \subseteq \Mat_X(\cx)$ contains all the 
edge weights under standard naming conventions for $A_0,\ldots,A_d,E_0,\ldots,E_d$.

The most basic examples of scaffold identities are easily verified: \\
\input{arxivdiagrams/vi.tex} \\
where $v_i=p_{ii}^0$ is the valency of the graph corresponding to the $i^{\rm th}$ basis relation 
$R_i$ of the scheme, and \\
\input{arxivdiagrams/firstpijk.tex} \\
where $p_{ij}^k$ is the intersection number defined above.  The first fundamental relations we encounter are given by the following lemma, recorded in our appendix as rules $\sSR{2}$ and $\sSR{2'}$, respectively. 
The first of these statements is obvious and one may derive short proofs for the second \cite{cgs}. But we will defer these  proofs to Theorem \ref{Thm:DeltaWyebases}  where more is accomplished.

\begin{lemma}
\label{Lem:pijkqijk}
For any association scheme $(X,\cR)$ with intersection numbers $p_{ij}^k$ and Krein parameters
$q_{ij}^k$, we have $p_{ij}^k = 0$ if and only if 
\begin{center}
\input{arxivdiagrams/pijk.tex}
\end{center}
and $q_{ij}^k=0$ if and only if 
\begin{center}
\input{arxivdiagrams/qijk.tex}
\end{center}
\end{lemma}

We also make frequent use of the following lemma.

\begin{lemma}
\label{Lem:EiEjEksum}
Let $(X,\cR)$ be an association scheme with Bose-Mesner algebra having primitive idempotents
$\{E_0,\ldots,E_d\}$ where, for $0\le h\le d$, $h'$ is the index for which $E_{h'} = E_h^\top$. Then \\
\begin{center}
\input{arxivdiagrams/EiEjEkTripleProduct.tex}
\end{center}
In particular, the zeroth order scaffold on the left is zero if and only if 
$q_{ij}^{k'} = 0$.
\end{lemma}

\begin{proof}
We compute
\begin{center}
\input{arxivdiagrams/BasicTripleProducts.tex} 
\end{center}
and note that $\tr E_{k'} = \tr E_k = m_k$. By the same token, 
\begin{center}
\input{arxivdiagrams/KreinParameterParallel.tex}  
\end{center}
the three possible evaluations all seen to be equal using the basic identity
$ q_{ij}^{k'} m_k = q_{ik}^{j'} m_j $.  $\Box$
\end{proof}

%
%
%
\subsection{Cometric association schemes}
\label{Subsec:cometric}

The present project was inspired by two papers of Hiroshi Suzuki \cite{suztwoq,suzimprim} establishing
fundamental structural properties of  cometric 
association schemes. In this section, we assume that $(X,\cR)$ is a symmetric association scheme
and $E_0,\ldots, E_d$ is the standard basis of primitive idempotents of its Bose-Mesner algebra
with Krein parameters $q_{ij}^k$ defined by the equations
$$ E_i \circ E_j = \frac{1}{|X|} \sum_{k=0}^d q_{ij}^k E_k. $$

\subsubsection{The Pinched Star and the Hollowed Delta}
\label{Subsubsec:toolsforcometric}

Since 
$$ E_i \circ E_j = \frac{1}{|X|} \sum_{h=0}^d q_{ij}^h E_h $$
we have 
\begin{equation}
\label{Eq:drumstick}
 \left(  E_i \circ E_j \right) E_k = \frac{q_{ij}^k}{|X|} E_k . 
\end{equation}
So have Rule $\sSR{3}$:

\input{arxivdiagrams/drumstick.tex} 

\medskip

and 

\input{arxivdiagrams/zerodrumstick.tex} 

\medskip

\noindent
Our notation here implies that no edges 
of $G$ are incident to the center node in the left diagram other than the three edges shown. We refer to the 
diagram on the left as the \emph{pinched star}.

\bigskip

Dual to this is the following identity for intersection numbers: since 
\begin{equation}
\label{Eq:hollowtriangle}
(A_j A_k) \circ A_i = \left( \sum_{e=0}^d p_{jk}^e A_e \right) \circ A_i = p_{jk}^i A_i , 
\end{equation}
we have the following collapse of the hollowed Delta, denoted $\sSR{3'}$ in the appendix:

\input{arxivdiagrams/wing.tex}

\subsubsection{Isthmuses}
\label{Subsubsec:isthmus}

The fundamental identity of Cameron, Goethals and Seidel \cite{cgs} given in Equation {\bf (2)} at the  beginning of
Section \ref{Sec:schemes} extends, using Rule $\sSR{0}$ to give us information about scaffolds of higher order.
Suzuki \cite[Lemma 4]{suzimprim} proved the symmetric version of the following ``Isthmus Lemma'',  
based on ideas of Dickie (Cf. \cite[Lemma 4.2.2]{dickie}). We extend Suzuki's result to the case of
(commutative) association schemes using scaffold notation and we denote this as Rule $\sSR{4}$.

\bigskip

\begin{lemma}
\label{Lem:Isthmus}
Let $(X,\cR)$ be an association scheme.
\end{lemma}
\noindent {\bf (I)} {\it If $q_{jk}^e \cdot q_{\ell m}^e =0$ for all $e\neq h$, then}

\medskip

\input{arxivdiagrams/SuzukiIsthmus1.tex} 

\medskip

\noindent {\bf (II)} {\it If $q_{jk}^e \cdot q_{\ell m}^e =0$ for all $e\neq h$, then}

\medskip

\input{arxivdiagrams/SuzukiIsthmus2.tex} 

\begin{proof} We provide a proof of (II) only:

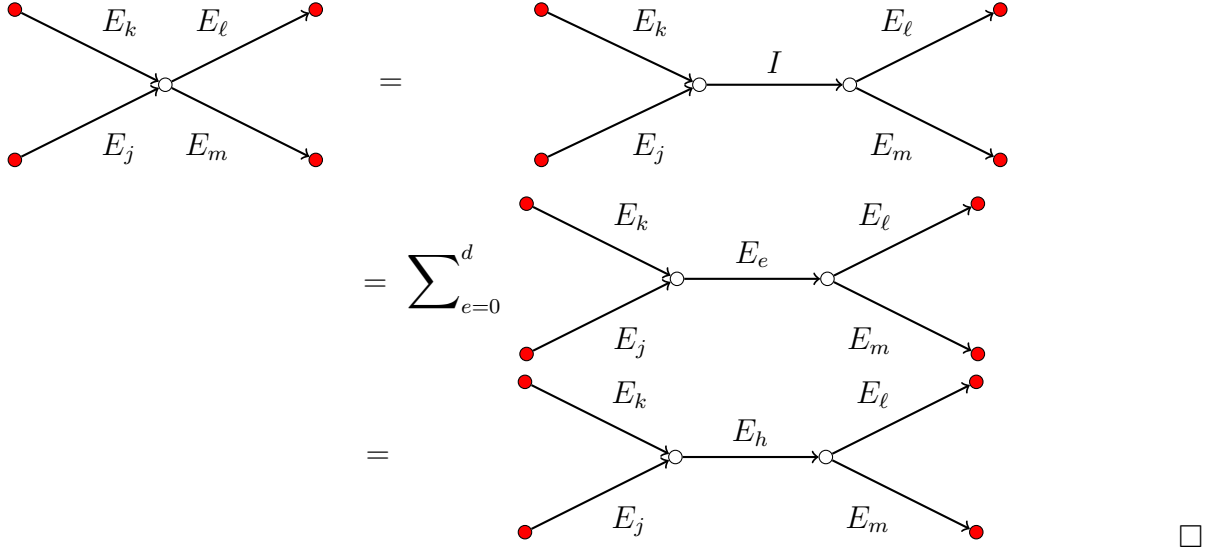
\input{arxivdiagrams/Isthmusproof1.tex}  \medskip

\hfill \hfill
\input{arxivdiagrams/Isthmusproof2.tex}  \hfill $\phantom{\Box}$

\hfill \hfill
\input{arxivdiagrams/Isthmusproof3.tex}  \hfill $\Box$
\end{proof}

For example, in any cometric association scheme with $Q$-polynomial ordering $E_0,E_1,E_2,\ldots$, we have
\input{arxivdiagrams/FourEVanish.tex}  $\displaystyle = \mathbf{0}$ whenever $|k-j| > h+i$ and 
\input{arxivdiagrams/QbipEstar.tex} \!\!\! $\displaystyle = \mathbf{0}$ for any $Q$-bipartite association scheme\footnote{A cometric 
association scheme with $Q$-polynomial ordering $E_0,\ldots,E_d$ is $Q$-\emph{bipartite} with respect to this 
ordering if $q_{ij}^k=0$ whenever $i+j+k$ is odd.} in which $j_1+ \cdots + j_\ell$ is odd.
On the other hand, there are cases where we may easily see that such a scaffold is non-zero. It is easy to see that scaffolds of the form \input{arxivdiagrams/NonzeroAstar.tex} and scaffolds of the form \input{arxivdiagrams/NonzeroEcycle.tex} can never be zero.

\begin{proposition}
\label{Prop:NonzeroEstar}
Let $(X,\cR)$ be an association scheme  
having basis of primitive idempotents
$E_0,E_1, \ldots, E_d$ and second eigenmatrix $Q$.  If $j_1,\ldots, j_\ell \in \{ 0,1,\ldots,d\}$ satisfy
$ Q_{i j_1} \cdots Q_{i j_\ell} \ge 0 $
for all $0\le i \le d$, then 

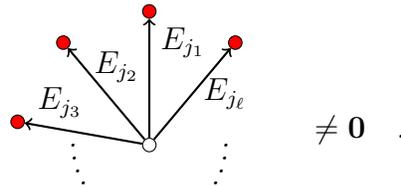
\vspace*{-0.25in}
\begin{center}
\input{arxivdiagrams/NonzeroEstar.tex}
\end{center}
\end{proposition}

\begin{proof}
Let $y\in X$. We simply show that the coefficient of $\hat{y} \otimes \hat{y} \otimes \cdots \otimes \hat{y}$ is strictly positive. This coefficient is 
$$\sum_{x\in X} \left( E_{j_1} \right)_{x,y} \left( E_{j_2} \right)_{x,y}  \cdots \left( E_{j_\ell} \right)_{x,y} = \frac{1}{|X|^\ell}
\left( Q_{0, j_1}  \cdots Q_{0, j_\ell} + \sum_{i=1}^d v_i Q_{i j_1} \cdots Q_{i j_\ell}  \right) > 0. \quad \Box$$
\end{proof}

The author does not know of a nonzero scaffold of this form in which the coefficients of $\hat{y} \otimes \hat{y} \otimes \cdots \otimes \hat{y}$ are all zero.

The dual to Lemma \ref{Lem:Isthmus}, denoted $\sSR{4'}$,  is given without proof.

\begin{lemma}
\label{Lem:DualIsthmus}
Assume $(X,\cR)$ is an association scheme.

\bigskip

\noindent {\bf (I)} If $p_{hi}^e \cdot p_{jk}^e =0$ for all $e\neq \ell$, then

\medskip

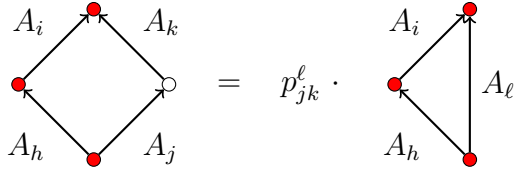
\input{arxivdiagrams/DualIsthmus1.tex}

\bigskip

\noindent {\bf (II)} If $p_{hi}^e \cdot p_{jk}^e =0$ for all $e\neq \ell$, then

\medskip

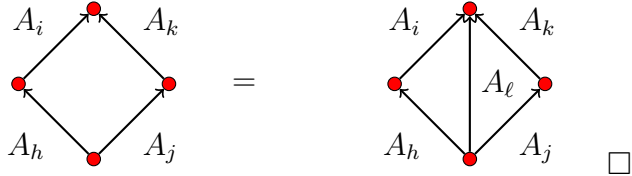
\input{arxivdiagrams/DualIsthmus2.tex} \ \ $\Box$
\end{lemma}

\subsubsection{Dickie's Theorem}  
\label{Subsubsec:dickie}

Let  $(X,\cR)$ be a symmetric association scheme with Bose-Mesner algebra $\BMA$. An ordering 
$E_0,E_1,\ldots,E_d$ of its basis of primitive idempotents is a \emph{cometric ($Q$-polynomial) ordering} 
if the following conditions are satisfied:
\begin{itemize}
\item $q_{ij}^k=0$ whenever any one of the indices $i,j,k$ exceeds the sum of the remaining two, and
\item  $q_{ij}^k > 0$ when $i,j,k\in \{0,1,\ldots,d \}$ and any one of the indices equals the sum of the remaining two.
\end{itemize}
We say that $(X,\cR)$ is a \emph{cometric} (or $Q$-\emph{polynomial}) association scheme when such an ordering
exists. Now suppose that $E_0,E_1,\ldots, E_d$ is a cometric ordering and recall the standard  abbreviations 
for cometric scheme parameters 
$$ a_j^* = q_{1,j}^j \qquad b_j^* = q_{1,j+1}^j \qquad c_j^* = q_{1,j-1}^j  ~ .$$
Note that $b_j^* > 0$ for $0\le j < d$ and $c_j^* > 0$ for $1\le j \le d$.

To illustrate  scaffold notation, we now give two versions of G.\ Dickie's proof  of a theorem 
from his 1995 dissertation \cite{dickie}.

\begin{theorem}[Dickie, Thm.~4.1.1]
Suppose $(X,\cR)$ is a cometric association scheme with $Q$-polynomial ordering 
$$ E_0, E_1, \ldots, E_d.$$
If $0<j<d$ and  $a_j^*=0$, then $a_1^*=0$. 
\end{theorem}

\begin{proof}
We begin with our assumption: since $ 0=a_j^* = q_{1,j}^{j}$, we have, from {\bf (2)} above,
\begin{equation} \label{EDckKrein1}
 \zero  = \sum_{a,b,c,z\in X} (E_j)_{a,z} (E_{j})_{b,z} (E_{1})_{c,z} \hat{a} \otimes \hat{b} \otimes \hat{c}~.
\end{equation}
which is more traditionally expressed as 
$$  \sum_{z\in X} (E_j)_{a,z} (E_{j})_{b,z} (E_{1})_{c,z} = 0 \qquad \left( \forall a,b,c \in X\right).$$
Since $\displaystyle E_j(E_1 \circ E_{j+1})  =     \frac{q_{1,j+1}^j }{|X|} E_j$ with $q_{1,j+1}^j \neq 0$, we then have, for all $a,b,c \in X$, 
\begin{eqnarray*}
 0 &=& \sum_{z \in X} \left( E_j(E_1 \circ E_{j+1})  \right)_{a,z}  (E_{j})_{b,z} (E_{1})_{c,z} \\
    &=& \sum_{x,z \in X}  (E_j)_{a,x} (E_1 \circ E_{j+1})_{x,z}  (E_{j})_{b,z} (E_{1})_{c,z}  \\
    &=& \sum_{x,z \in X}  (E_j)_{a,x} (E_1)_{x,z} ( E_{j+1})_{x,z}  (E_{j})_{b,z} (E_{1})_{c,z}  ~ .
\end{eqnarray*}

Now we apply the Isthmus Lemma with the observation that  $\displaystyle  q_{j,1}^e \cdot q_{1,j+1}^e = 0$
for any $e\neq j+1$: 
$$  0 =  \sum_{x,y,z \in X}  (E_j)_{a,x} (E_1)_{x,y} ( E_{j+1})_{x,z}  (E_{j})_{b,y} (E_{j+1})_{y,z} (E_{1})_{c,z} ~.$$
This holds true for \underline{any} choice of vertices $a,b,c$. So we can scale each vanishing quantity 
by a value dependent on $a,b,c$, such as 
$$   (E_{j-1} )_{a,b}  (E_1)_{b,c} (E_{1})_{c,a} ,$$
and still obtain a zero result. We sum this product over all choices of $a$, $b$ and $c$ to find:
\begin{equation}  \label{EDckPrism}
  0 =   \sum_{a,b,c,x,y,z \in X}  (E_{j-1} )_{a,b}  (E_1)_{b,c} (E_{1})_{c,a} 
(E_j)_{a,x} (E_1)_{x,y} ( E_{j+1})_{x,z}  (E_{j})_{b,y} (E_{j+1})_{y,z} (E_{1})_{c,z} ~.
\end{equation}

Now it's time to eliminate some factors from these huge products. We apply the Isthmus Lemma. Observe that
$\displaystyle  q_{1,j-1}^e \cdot q_{1,j+1}^e = 0$ for any $e\neq j$. So we can identify $b$ and $y$ in (\ref{EDckPrism}):
\begin{equation} \label{EDck1st-contract}
  0 =   \sum_{a,c,x,y,z \in X}  (E_{j-1} )_{a,y}  (E_1)_{y,c} (E_{1})_{c,a} 
(E_j)_{a,x} (E_1)_{x,y} ( E_{j+1})_{x,z} (E_{j+1})_{y,z} (E_{1})_{c,z} ~ .
\end{equation}

\bigskip

Next we apply the Isthmus Lemma to identify $a$ and $x$ in (\ref{EDck1st-contract}). Again using the 
fact  that $\displaystyle  q_{1,j-1}^e \cdot q_{1,j+1}^e = 0$ for any $e\neq j$, we may write 
\begin{equation} \label{EDck2nd-contract}
  0 =   \sum_{c,x,y,z \in X}  (E_{j-1} )_{x,y}  (E_1)_{y,c} (E_{1})_{c,x} 
(E_1)_{x,y} ( E_{j+1})_{x,z} (E_{j+1})_{y,z} (E_{1})_{c,z} ~ .
\end{equation}

We can replace $ (E_1)_{x,y} (E_{j-1} )_{x,y}$ in (\ref{EDck2nd-contract}) by $(E_1 \circ E_{j-1} )_{x,y} $ to obtain
\begin{equation} \label{EDicbubbletri}
 0 =   \sum_{c,x,y,z \in X}  (E_1 \circ E_{j-1} )_{x,y}  (E_1)_{y,c} (E_{1})_{c,x} 
( E_{j+1})_{x,z} (E_{j+1})_{y,z} (E_{1})_{c,z} ~ .
\end{equation}
Now we expand
$$ E_1 \circ E_{j-1} = \frac{b_{j-2}^* }{|X|} E_{j-2} +  \frac{a_{j-1}^* }{|X|} E_{j-1} +  \frac{c_{j}^* }{|X|} E_{j} $$
and observe $ q_{1,j+1}^e = 0  $ for $e< j$.  Since $c_j^*\neq 0$, we have
\begin{eqnarray*}
 0 &=&  \sum_{e=0}^d   \frac{q_{1,j+1}^e }{|X|}  \sum_{c,x,y,z \in X}  (E_e)_{x,y}  (E_1)_{y,c} (E_{1})_{c,x} 
( E_{j+1})_{x,z} (E_{j+1})_{y,z} (E_1)_{c,z} \\
  &=&  \sum_{c,x,y,z \in X}  (E_{j})_{x,y}  (E_1)_{y,c} (E_{1})_{c,x}  ( E_{j+1})_{x,z} (E_{j+1})_{y,z} (E_{1})_{c,z} 
\end{eqnarray*}
since $q_{1,j+1}^{j}=b_j^* \neq 0$ by the cometric property.

We apply the Isthmus Lemma on the $(E_{j+1})_{y,z}$ factor, observing that
$\displaystyle  q_{1,j+1}^e \cdot q_{1, j}^e = 0$ for any $e\neq j+1$. This gives
$$ 0 =  \sum_{c,x,z \in X} (E_{j})_{x,z}  (E_1)_{z,c} (E_{1})_{c,x}  ( E_{j+1})_{x,z}  (E_{1})_{c,z} ~ .$$
Let's re-organize the terms
$$ 0 =  \sum_{c,z\in X}   (E_1)_{c,z}^2
 \sum_{x\in X} (E_{j})_{x,z}   (E_{j+1})_{x,z}  ( E_{1})_{x,c}  ~. $$

\noindent We can again apply (\ref{Eq:drumstick}).  First use the entrywise product to write this as 
$$ 0 =  \sum_{c,z\in X}   (E_1)_{c,z}^2
 \sum_{x\in X} (E_{1})_{c,x}   (E_{j} \circ E_{j+1})_{x,z}  ~. $$
Now, using (\ref{Eq:drumstick}) and cancelling $q_{j,j+1}^1/|X| \neq 0$,  we may write
$$ 0 =  \sum_{c,z\in X}   (E_1)_{c,z}^3  ~. $$
In other words, $\SUM( E_1 \circ E_1 \circ E_1 ) = 0$
which is equivalent to $a_1^*=0$, our desired result.
$\Box$
\end{proof}

\bigskip

This second version of the same proof  indicates that  Dickie used something equivalent to scaffolds:

\bigskip

\begin{proof}
This proof consists almost entirely of a sequence of 
scaffolds all equal to the zero tensor $\zero \in V^{\otimes 3}$ or the zero scalar. 
We begin with our assumption that $a_j^*=0$:

\input{arxivdiagrams/ajstarequalszero.tex}

\input{arxivdiagrams/dickieadddrumstick.tex}

\input{arxivdiagrams/dickiefirstisthmus.tex}

Now take the scalar product of this vanishing scaffold with the third-order tensor  \\
\input{arxivdiagrams/dickielittletriangle.tex} 

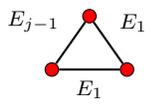
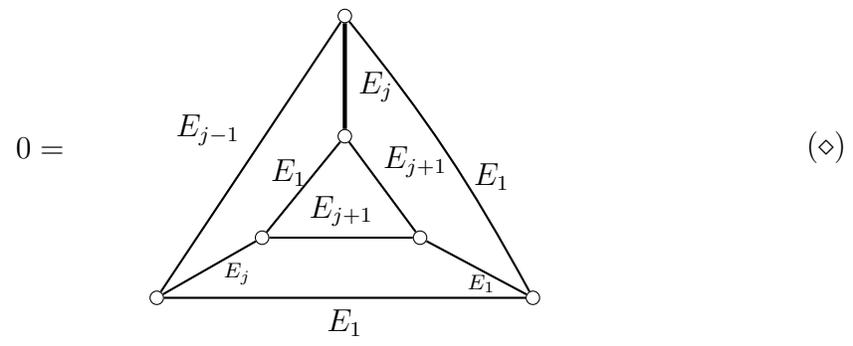
\begin{center}
\input{arxivdiagrams/dickieprism.tex} 
\end{center}
(Note that here and below we use a bold edge to indicate the location in the diagram where the next simplification will be applied.)

\input{arxivdiagrams/dickiecontractverticaledge.tex}

\input{arxivdiagrams/dickiecontractleftedge.tex}

Using the entrywise product, \\ 
\begin{center}
\input{arxivdiagrams/dickiereplacedoubleedge.tex}
\end{center}

Now we expand
$$ E_1 \circ E_{j-1} = \frac{b_{j-2}^* }{|X|} E_{j-2} +  \frac{a_{j-1}^* }{|X|} E_{j-1} +  \frac{c_{j}^* }{|X|} E_{j} $$
and observe $ q_{1,j+1}^e = 0  $ for $e< j$. Since $c_j^*\neq 0$, we have 

\input{arxivdiagrams/dickieSchursimplify.tex}

\input{arxivdiagrams/dickiecontractnortheast.tex}  

\bigskip

Now we have a pinched star! \input{arxivdiagrams/dickiedrumstick.tex}  and we know
$$ (E_j \circ E_{j+1} ) E_1 = \frac{ q_{j,j+1}^1 }{|X|} E_1 $$
with $ q_{j,j+1}^1\neq 0$ by the cometric property.
So we have 

 \input{arxivdiagrams/dickiecontractdrumstick.tex} 

That is, $\SUM( E_1 \circ E_1 \circ E_1 ) = 0$ which tells us that $q_{11}^1=0$, or $a_1^*=0$. $\Box$
\end{proof}


\subsubsection{Suzuki's Theorem}
\label{Subsubsec:suzuki}

We find the same sort of proof structure in Proposition 3 of Suzuki's paper \cite{suzimprim}.
 
 \begin{theorem}[Suzuki]
 \label{Thm:Suzuki}
 In a cometric association scheme $(X,\cR)$ with $Q$-polynomial ordering $E_0,\ldots, E_d$ of its primitive 
 idempotents, if indices $j\le i \le i+j\le h+i+j\le d$ satisfy $q_{j,h+i}^e \cdot q_{i-j,h+j}^e = 0$ for 
 all $e\neq h+i-j$ and $q_{i,h+j}^{h+i}=0$,  then $q_{j,h+j}^{h+j}=0$.
 \end{theorem}
 
 Since the proof structure is similar to the proof of Dickie's Theorem above, we simply present one diagram 
 (a generalization of $(\diamond)$ above) and leave the steps to the reader.
 
 Starting with the equation $q_{i,h+j}^{h+i}=0$, the proof consists in constructing a third order tensor which, 
 under the given hypotheses, must be zero.  We then compute the scalar product of this scaffold 
 with \input{arxivdiagrams/SuzProof3a.tex}  to obtain the identity
 
 \begin{center}
 \input{arxivdiagrams/SuzProofnutshell.tex}
 \end{center}

In \cite{suztwoq}, 
Suzuki uses these and other ideas to prove that a cometric association scheme which is not 
a polygon admits at most two $Q$-polynomial orderings and to narrow down the possibilities for a 
second $Q$-polynomial ordering given that $\{E_0,E_1,\ldots,E_d\}$ is such an ordering. 
Here is a lemma from that paper.

\begin{lemma}[Suzuki]
Let $(X,\cR)$ be a cometric association scheme with $Q$-polynomial ordering $E_0,E_1,\ldots,E_d$ of its
primitive idempotents. Let $h,i,j,k,\ell,m$ be indices satisfying $q_{i,j}^h\neq 0$ and
$$ \left( \forall e\neq \ell \right) \left( q_{h,m}^e \cdot q_{i,k}^e = 0 \right) \qquad \mbox{and} \qquad 
 \left( \forall e\neq m \right) \left( q_{h,\ell}^e \cdot q_{j,k}^e = 0 \right) . $$
 Then $q_{k,\ell}^i = q_{k,m}^j$.  $\Box$
\end{lemma}

To prove this, Suzuki shows two ways to manipulate a scaffold built on $K_4$:

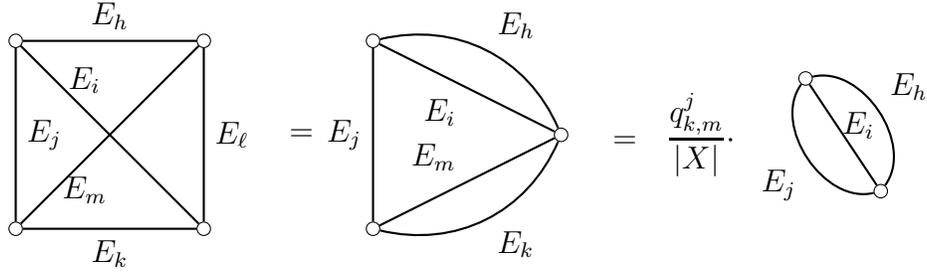
\begin{center}
\input{arxivdiagrams/SuzukiK4collapse1.tex}
\end{center} 
and, at the same time,
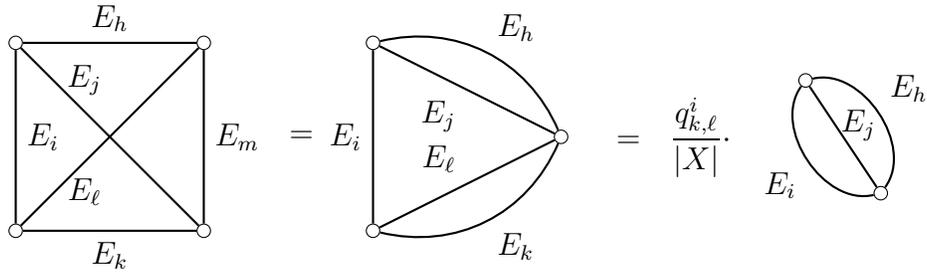
\begin{center}
\input{arxivdiagrams/SuzukiK4collapse2.tex}
\end{center}

The dual argument is well known in the theory of distance-regular graphs. 
While there seem to be many identities involving parameters of association schemes that are probably 
best proven by more elementary means, we give one fundamental identity from the book of Brouwer, Cohen
and Neumaier \cite[Lemma 2.1.1]{bcn} here, 
along with the diagrams involved in proving the identity:
$$ \sum_{\ell=0}^d  p_{hi}^\ell p_{\ell j}^k = \sum_{r=0}^d p_{hr}^k p_{ij}^r 
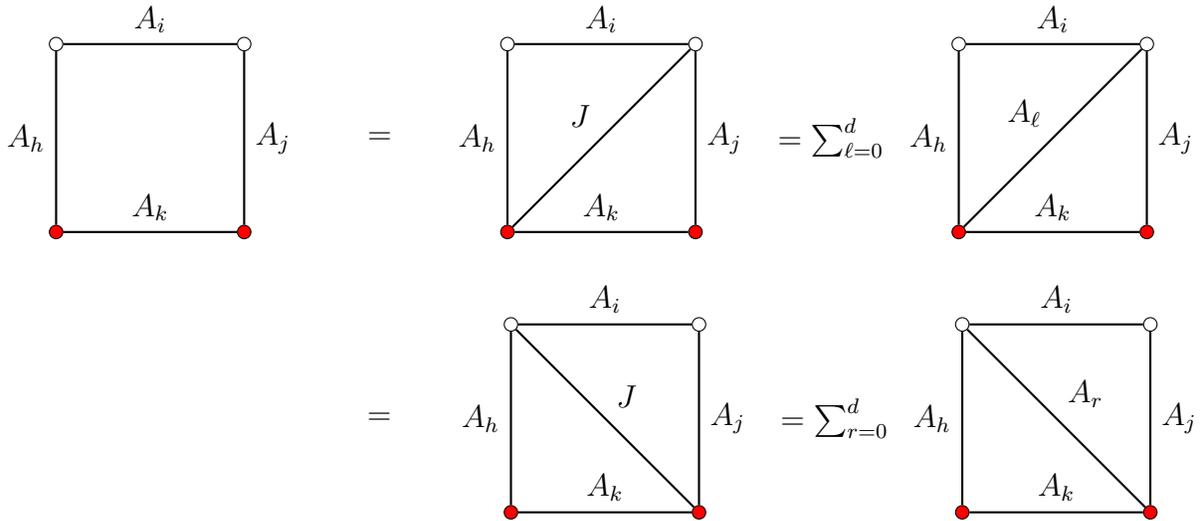
$$

\begin{center}
\input{arxivdiagrams/doublecountK4.tex}
\end{center}

\begin{center}
\input{arxivdiagrams/seconddoublecountK4.tex}
\end{center}
 
 %
%
%
\subsection{Generalized intersection numbers for distance-regular graphs}
\label{Subsec:intersection}

Following Coolsaet and Juri\v{s}i\'{c} \cite{cooljur}, we now define generalized intersection
numbers for an arbitrary association scheme. 
For $a_1,\ldots,a_k\in X$ and $i_1,\ldots,i_k \in \{0,\ldots,d\}$, 
define
$$ \left[ \begin{array}{rrrr}  a_1 & a_2 & \cdots & a_k \\ i_1 & i_2 & \cdots & i_k \end{array} \right]
= \left| \left\{ b\in X \mid (\forall j)\left( (a_j,b)\in R_{i_j} \right) \right\} \right| ;$$
this is the number of vertices $i_j$-related to $a_j$ for $j=1,\ldots,k$. One immediately notices these values
as the coefficients of elementary tensors in the following star scaffold:
\begin{center}
\hfill \input{arxivdiagrams/GenIntersectionNums.tex}
\end{center}
In some cases --- e.g., when $k$ is small --- this number does not depend  on the choice of the vertices $a_j$ but only on the $\binom{k}{2}$
relations joining them. For instance, here are the encoding of the most basic forms in the language of 
scaffolds:

\input{arxivdiagrams/TrivialTripleIntersectionNumbers.tex} 

Triple intersection numbers are a powerful tool in proving non-existence of 
certain association schemes \cite{cooljur,gavrsudavida}. The scalars 
$\left[ \begin{array}{rrr}  a & b & c \\ i & j & k \end{array} \right]$ are simply the
coefficients in the expansion of one of our fundamental ``star'' diagrams in the basis
of simple tensors. We have

\hfill \input{arxivdiagrams/TripleIntersectionNums.tex} \hfill \break

Consider a distance-regular graph $\Gamma$ of diameter $d$ with vertex set $X$ and path-length metric
$\partial$. Then
$$ \left[ \begin{array}{rrrr}  a_1 & a_2 & \cdots & a_k \\ i_1 & i_2 & \cdots & i_k \end{array} \right]
= \left| \left\{ b\in X \mid \partial(b,a_1)=i_1, \ \ldots,  \partial(b,a_k)=i_k \right\} \right| $$
counts vertices $b$ simultaneously at distance $i_j$ from basepoint $a_j$ for $1\le j \le k$.

Coolsaet and Juri\v{s}i\'{c} \cite{cooljur} observe, in their Equation (5), that

\input{arxivdiagrams/CJ0jk.tex}  %

Their Equation (6) can be expressed

\input{arxivdiagrams/CJsum.tex}  %

Of course, $\left[ \begin{array}{rrr}  a & b & c \\ i & j & k \end{array} \right]$ depends on the choice of $a$, $b$, and $c$, but several researchers (e.g.,  \cite{cooljur,gavrsudavida}) 
have ruled out feasible parameter sets for distance-regular graphs by analyzing linear relations that these
numbers must satisfy. A key insight in \cite{cooljur} is the following.  If we know that a Krein parameter $q_{r,s}^t$
vanishes, then we have
 \input{arxivdiagrams/qrstEqualZero.tex}

\noindent and expanding $E_r = \frac1{|X|} \sum_{i=0}^d Q_{ir} A_i$, and similarly for $E_s$ and $E_t$, we 
find

\input{arxivdiagrams/TripleIntersectionNumEquation.tex} 

\noindent The same logic works on the dualized diagrams to give us the following apparently new identity:

\input{arxivdiagrams/TripleKreinParamEquation.tex}

\section{Vector spaces of scaffolds}
\label{Sec:vectorspaces}

Given a subspace or subalgebra $\BMA$ of $\Mat_X(\cx)$, we now investigate various
spaces contained in the vector space spanned by all scaffolds with edge weights in $\BMA$. We
obtain interesting results when we fix the rooted diagram or simply fix the number of root
nodes.

%
%
%
\subsection{Inner products and bilinear maps}
\label{Subsec:innerprod}

Standard scalar products on tensors and certain natural bilinear maps can be described easily as 
gluing operations on diagrams.
Our most familiar inner product on tensors is the Frobenius product of two matrices:
$$ M = \sum_{a,b \in X} M_{ab} \ \hat{a} \otimes \hat{b}, \qquad  N = \sum_{a,b \in X} N_{ab} \ \hat{a} \otimes \hat{b},
 \qquad 
 \langle  M,N \rangle = \sum_{a,b \in X}  \bar{M}_{ab} N_{ab}~. $$
 In terms of scaffolds, this is expressed

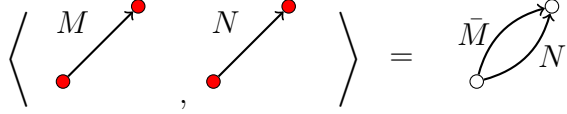
\input{arxivdiagrams/Frob1.tex}

\noindent where $\bar{M}$ is the matrix obtained by conjugating each entry of $M$. This is achieved by the following operation on scaffolds:

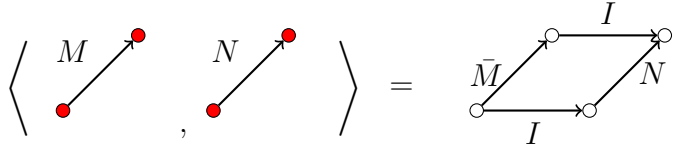
\input{arxivdiagrams/Frob2.tex} 

We extend this to scaffolds $\scafs$ and $\scaft$ 
of order $m$: we assume a consistent ordering of the root nodes in 
the two diagrams, indicated by their spatial arrangement, and simply join each pair of 
corresponding root nodes $v_s,v_t$ by an edge $e=(v_s,v_t)$ with weight $w(e)=I$, conjugate the edge weights coming from the left argument and, in the new scaffold, make all nodes hollow.  For instance,

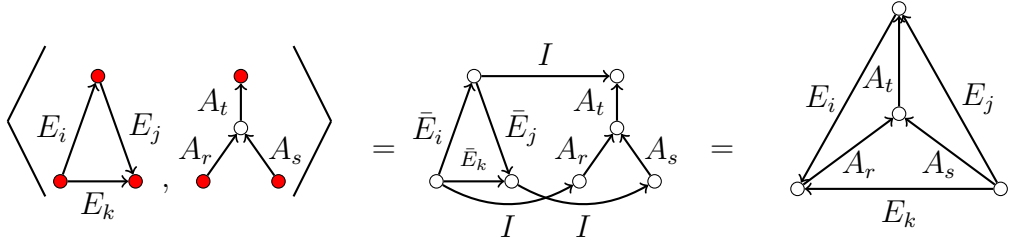
\input{arxivdiagrams/IPorder3.tex} \\
is simply the linear extension of the product defined by 
\begin{equation}
\label{Eqn:ip3}
\langle \mathbf{y}_1 \otimes \mathbf{y}_2 \otimes \cdots \otimes \mathbf{y}_m \ , \ \mathbf{z}_1 \otimes \mathbf{z}_2 \otimes \cdots \otimes \mathbf{z}_m \rangle
=
\langle \mathbf{y}_1, \mathbf{z}_1 \rangle
\langle \mathbf{y}_2 , \mathbf{z}_2 \rangle \cdots
\langle \mathbf{y}_m , \mathbf{z}_m \rangle
\end{equation}
for $\mathbf{y}_1,\mathbf{y}_2 , \ldots, \mathbf{y}_m , \mathbf{z}_1 ,\mathbf{z}_2 , \ldots ,  \mathbf{z}_m \in \cx^X$
in the case $m=3$.

\begin{example}
Let $X$ denote the vertex set of the Petersen graph with adjacency matrix $A_1$ and primitive idempotents $E_0,E_1,E_2$ satisfying $A_1 E_1 = E_1$ and $A_1 E_2 = -2E_2$. Although every edge weight below
is positive semidefinite, straightforward computation shows 
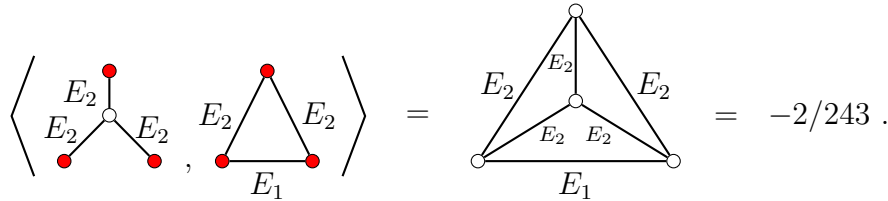
\begin{center}
\input{arxivdiagrams/IPcounteraxamplePet.tex}
\end{center}
\end{example}

Rather than take a proper scalar product, we may want to combine two scaffolds in such a way as to reduce the
overall order by pairing up selected  components of tensors. Let us begin with an example. If 
$$ \scafs = \sum_{a,b,c,d} M_{a,b,c,d}  \ \hat{a} \otimes \hat{b} \otimes \hat{c} $$
and 
$$ \scaft = \sum_{a,c,x,y} N_{a,c,x,y} \ \hat{a} \otimes \hat{c} \otimes \hat{y} $$
we desire notation for  bilinear maps such as 
$$ \langle \langle  \scafs,  \scaft \rangle  \rangle = 
\sum_{a,b,c,d,x,y} M_{a,b,c,d} N_{a,c,x,y} \  \hat{b} \otimes \hat{y} . $$
More generally, we define a bilinear map 
$$\langle\langle \cdot,\cdot \rangle  \rangle: \left( V^{\otimes r}\otimes V^{\otimes s}\right) \times
 \left( V^{\otimes r}\otimes V^{\otimes t}\right) \rightarrow V^{\otimes s}\otimes V^{\otimes t} $$
 via
 \begin{equation}
\label{Eqn:bilin1}
  \langle\langle \left( u_1 \otimes \cdots \otimes u_r \right) \otimes  \left( v_1 \otimes \cdots \otimes v_s \right)
 \  , \
  \left( u'_1 \otimes \cdots \otimes u'_r \right) \otimes  \left( w_1 \otimes \cdots \otimes w_t \right)   \rangle  \rangle
  = \hfill \phantom{X}$$
  $$ \phantom{X} \hfill \left( \prod_{i=1}^r \langle u_i,u'_i\rangle \right) \left( v_1 \otimes \cdots \otimes v_s \right) \otimes  \left( w_1 \otimes \cdots \otimes w_t \right) 
\end{equation}
  and extending linearly. Note that when $s=t=0$, this is the standard scalar 
  product on $r^{\rm th}$ order scaffolds and the unusual
  notation $\langle \langle  \cdot , \cdot \rangle \rangle$ specializes to $\langle  \cdot , \cdot \rangle$.
  
In the language of scaffolds, a bilinear map of this type can be achieved by 
joining corresponding root nodes with edges labelled with the identity matrix and changing
those nodes from solid to hollow (i.e., removing them from the set of root nodes).

Our convention for the inner product $\langle  \cdot , \cdot \rangle$ above required spatial alignment of 
corresponding nodes in the two diagrams. For these more general bilinear maps, we will adopt a 
different convention.
In order  to unambiguously indicate the pairing of root nodes on the left with root nodes on the right, we
now adopt the convention of arranging the components of $V^{\otimes r}$ in a row (resembling a row vector)
at the top of the first argument and, in the second argument, in a column (resembling a column vector) at the left
of the second argument. So, with this convention, the proper notation for the above product is now clear without
the addition of node labels. If we arrange two scaffolds $\scafs$ and $\scaft$ with edge weights in $\Mat_X(\re)$
as 

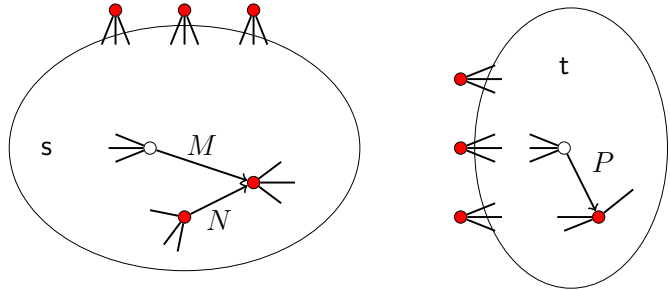
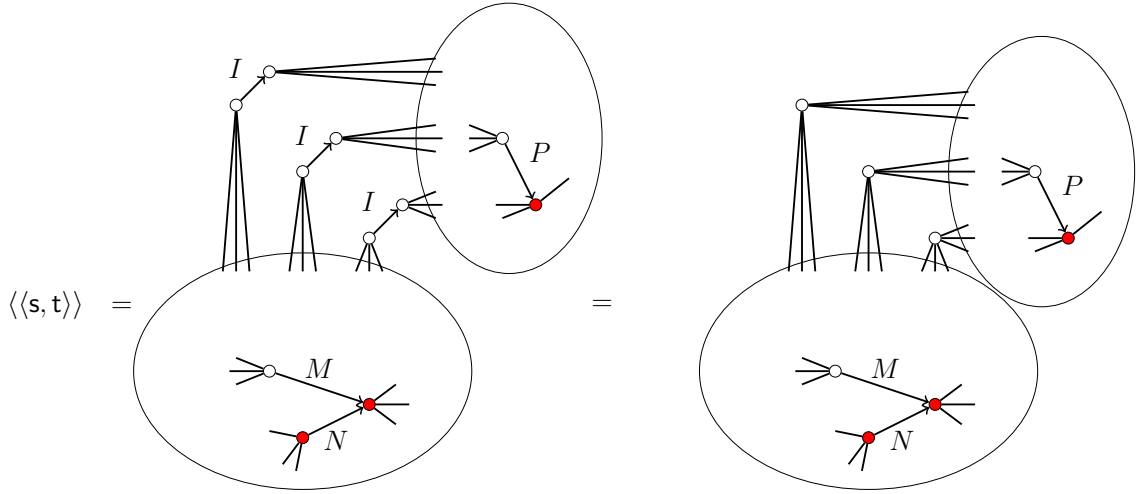
\begin{center}
\resizebox{3.5in}{!}{\input{arxivdiagrams/bilinear_product_1.tex}}
\end{center}

\noindent then we have 

\begin{center}
\resizebox{6in}{!}{\input{arxivdiagrams/bilinear_product_2.tex}}
\end{center}

In addition to scalar products, ordinary matrix multiplication may also be viewed as a special case: \\
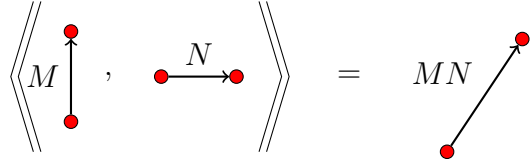
\begin{center}
\input{arxivdiagrams/bilinear_product_3.tex}
\end{center}

Another use of these bilinear maps arises when the second argument $\scaft$ in the product
$\langle \langle  \scafs,  \scaft \rangle  \rangle$ is simply a column of $r$ root nodes in an edgeless diagram.

An \emph{order reduction} operation $\hollow_{R'}$ replaces 
$$\sS = \sS(G , R ; w; F,\varphi_0 ) = \sum_{\substack{ \varphi:V(G) \rightarrow X\\
(\forall a\in F)(\varphi(a)=\varphi_0(a) )}} \quad \prod_{\substack{ e\in E(G)  \\ e=(a,b) }} w(e)_{\varphi(a),\varphi(b)} \bigotimes_{r\in R} \widehat{\varphi(r)} $$
by
$$\hollow_{R'} \left( \sS \right) = \sS(G , R' ; w; F,\varphi_0 ) = \sum_{\substack{ \varphi:V(G) \rightarrow X\\
(\forall a\in F)(\varphi(a)=\varphi_0(a) )}} \quad \prod_{\substack{ e\in E(G)  \\ e=(a,b) }} w(e)_{\varphi(a),\varphi(b)} \bigotimes_{r\in R'} \widehat{\varphi(r)}. $$
where $R' \subseteq R$. This maps $|R|^{\rm th}$ order tensors to tensors of order $|R'|$.  
In terms of diagrams,  the solid nodes outside $R'$ are  converted to hollow nodes.  We denote this as Rule 
$\sSR{10}$.

\begin{lemma}
\label{Lem:orderred}
If $\sS(G_1, R; w_1; F_1, \phi_1) = \sS(G_2, R;w_2; F_2, 
\phi_2)$  (defined on the same set $R$ of roots), then, for any $R'\subseteq R$, 
$\sS(G_1, R'; w_1; F_1, \phi_1) = \sS(G_2, R';w_2; F_2,  \phi_2)$. In particular,
if  $\sS(G, R;w; F, \phi_0) = \zero$  and $R'\subseteq R$, then $\sS(G, R';w; F, \phi_0) = \zero$ also. $\Box$
\end{lemma}

We may also apply such an operator to rooted diagrams (without edge 
labels) and view order reduction as a map from scaffolds define on $(G,R)$ to scaffolds on $(G,R')$ 
sending $\sS$ to $\hollow_{R'} \left( \sS \right)$.

In this subsection, we have indicated a variety of ways in which one might define a product, of some sort, on two scaffolds. In Section \ref{Subsec:Ter} below, we will discuss a less natural product which is nonetheless important 
in the theory of association schemes, namely the Terwilliger algebra product.  One may also define products on
triples, as Mesner and Bhattacharya did in 1990   \cite{MesnerBhatt} when they introduced ``association
schemes on triples''.  Their non-associative ternary product (Equation (1) on p209) can be viewed as the linear 
extension of a product written in scaffold notation as 

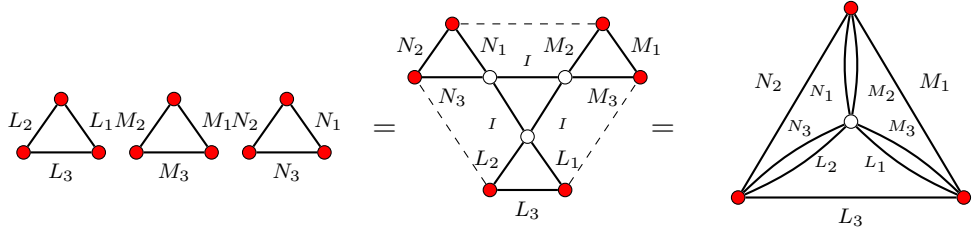
\begin{center}
\input{arxivdiagrams/MesnerBhattacharya.tex}
\end{center}

\subsection{Important subspaces}
\label{Subsec:vecspace}

This section can be viewed as a natural extension of some work of Terwilliger \cite{terPQ}. 

Let $(X,\cR)$ be a (commutative) association scheme with standard module $V= \cx^X$ and automorphism group\footnote{The automorphism group of an association scheme $(X,\{R_0,\ldots,R_d\}$ is defined to
be the subgroup of $\Sym(X)$ which preserves all relations $R_i$, $1\le i\le d$.} $\Sigma$. The group $\Sigma$
acts componentwise on elementary basis tensors $\hat{x}_1 \otimes \cdots \otimes \hat{x}_m$: if $\sigma \in \Sigma$
sends $x\in X$ to $x^\sigma$, then 
$$ \sigma : \hat{x}_1 \otimes \cdots \otimes \hat{x}_m \mapsto \widehat{x^\sigma_1} 
\otimes \cdots \otimes \widehat{x^\sigma_m}. $$
Each scaffold of order $m$ is an element of the tensor product $V^{\otimes m}$ and it is easy to see that, in full generality, the $m^{\rm th}$ order scaffolds span this space. But the space spanned by the symmetric scaffolds is 
typically smaller.  The vector space of  $m^{\rm th}$ order symmetric scaffolds contains an ascending chain 
of subspaces $S_m \subseteq  S_{m+1} \subseteq  S_{m+2} \subseteq \cdots$ where $S_t$ is the 
vector space spanned by $m^{\rm th}$ order symmetric scaffolds on $t$ nodes.

\begin{theorem}
\label{Thm:orbits}
Let $(X,\cR)$ be a symmetric association scheme with Bose-Mesner algebra $\BMA$ and automorphism 
group $\Sigma$.
The vector space $\sbW( m ; \BMA)$ of all linear combinations of $m^{\rm th}$ order scaffolds with edge weights in
$\BMA$ has dimension equal to the number of 
orbits of $\Sigma$ on ordered $m$-tuples of vertices. If $\cO_1,\ldots,\cO_N$ is a full list of 
orbits of $\Sigma$ on ordered $m$-tuples of vertices, then the tensors
$$\left\{  \sum_{(x_1,\ldots,x_m) \in \cO_h} \hat{x}_1 \otimes \cdots \otimes \hat{x}_m \middle\vert 1 \le h \le N \right\}$$
form a basis for this vector space.
\end{theorem}

\begin{proof}
It is not hard to see that every element of $\Sigma$ preserves every symmetric scaffold:
$$\hspace*{-3cm}  \sS(G , R ; w) = \sum_{ \varphi:V(G) \rightarrow X}\  \left( \prod_{\substack{ e\in E(G)  \\ e=(a,b) }} 
w(e)_{\varphi(a),\varphi(b)}  \right) 
 \widehat{\varphi(r_1)} \otimes  \widehat{\varphi(r_2)} \otimes \cdots \otimes  \widehat{\varphi(r_m)}
 =  $$
 $$\hspace*{4cm}    \sum_{ \varphi:V(G) \rightarrow X}\  \left( \prod_{\substack{ e\in E(G)  \\ e=(a,b) }} 
w(e)_{\varphi(a)^\sigma,\varphi(b)^\sigma}  \right) 
 \widehat{\varphi(r_1)^\sigma} \otimes  \widehat{\varphi(r_2)^\sigma} \otimes \cdots \otimes  
 \widehat{\varphi(r_m)^\sigma}. $$
 To see that the two spaces are equal, let $\cO_h$ be any orbit on $m$-tuples with orbit representative
 $(y_1,\ldots,y_m)$.  Let $G$ be the complete graph with vertex set $X$, edge weights $w(e)=A_i$ whenever
 $e=(a,b) \in R_i$ and root nodes $R=\{ y_1,\ldots,y_m\}$. Then
 $$ \sS(G , R ; w) =  \sum_{(x_1,\ldots,x_m) \in \cO_h} \hat{x}_1 \otimes \cdots \otimes \hat{x}_m. ~ \Box$$
\end{proof}
For example, if $(X,\cR)$ is \emph{Schurian} (i.e., there is a group $\Sigma$ acting on $X$ whose orbitals
are precisely $R_0,R_1,\ldots,R_d$), then the space of second order scaffolds is no larger than the 
Bose-Mesner algebra, which is the vector space of single-edge scaffolds of order two.

\bigskip

For each $m\ge 1$, we have various actions of $\Mat_X(\cx)$ on $V^{\otimes m}$ given
by Jaeger \cite{jaeger}. 
\begin{itemize}
\item[I] ({\sc node action}) For $A\in \Mat_X(\cx)$, $1\le i\le m$ and $\hat{x}_1 \otimes \cdots  
\otimes\hat{x}_m \in V^{\otimes m}$, 
define 
$$\cD^i_A : \hat{x}_1 \otimes \cdots   \otimes \hat{x}_i  \otimes \cdots  \otimes\hat{x}_m \  
\mapsto  \ \hat{x}_1 \otimes \cdots  \otimes A \hat{x}_i  \otimes \cdots  \otimes\hat{x}_m $$
Diagrammatically, this adds a node of degree one to a scaffold as follows: 
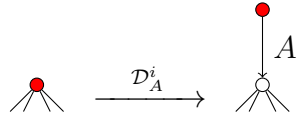
\begin{center}
\input{arxivdiagrams/MatrixMultVertex.tex} 
\end{center}
\item[II] ({\sc edge action}) For $A\in \Mat_X(\cx)$, $1\le i,j\le m$ and standard  basis element
$\hat{x}_1 \otimes \cdots  \otimes\hat{x}_m \in V^{\otimes m}$, define 
$$\cE^{i,j}_A : \hat{x}_1 \otimes \cdots   \otimes\hat{x}_m \  \mapsto  
\ \ A_{x_i,x_j} \, \hat{x}_1 \otimes \cdots  \otimes\hat{x}_m $$
\end{itemize}
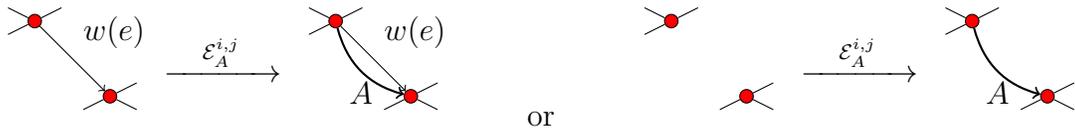
\begin{center}
 \input{arxivdiagrams/EdgeAction.tex} 
\end{center}

Using the notion of bilinear map discussed in the previous section, we can now use scaffold notation to re-state some 
results in Paul Terwilliger's paper \cite{terPQ}. Let $(X,\cR)$ be an association
scheme with Bose-Mesner algebra $\BMA$. The inner product space 
$$ V^{\otimes 3} = V\otimes V\otimes V = \spn \left\{ \hat{a} \otimes \hat{b} \otimes \hat{c} \mid a,b,c \in X \right\} $$
of all third order tensors can be viewed as an $\BMA^{\otimes 3}$-module in various ways as described 
in Section \ref{Subsec:basics}:
\begin{equation}
\label{Eqn:V3Daction}
 (M_1 \otimes M_2 \otimes M_3) \scafs = \left( \cD^1_{M_1} \circ \cD^2_{M_2} 
\circ \cD^3_{M_3} \right) (  \scafs ) , 
\end{equation}
where $\circ$ here denotes composition of functions, or
\begin{equation}
\label{Eqn:V3Eaction}
 \{M_1 \otimes M_2 \otimes M_3\} \scafs =  \left( \cE^{1,2}_{M_3} \circ \cE^{1,3}_{M_2} 
\circ \cE^{2,3}_{M_1}\right)  (  \scafs ) . 
\end{equation}
One must therefore be clear when discussing actions of this sort. For the moment, let us denote the action given
in (\ref{Eqn:V3Daction}) by $(\cdot)$ and  the action given in (\ref{Eqn:V3Eaction}) by $\{\cdot\}$.

While each of these actions modifies the diagram underlying a scaffold, we are also interested in operations
which leave the underlying rooted diagram unchanged.

\begin{definition}
\label{Def:Wspace}
Given a finite (di)graph $G=(V(G),E(G))$, a set of $m$ root nodes  $R \subseteq V(G)$, and  a 
vector subspace $\BMA$ of $\Mat_X(\cx)$, we denote by $\sbW( (G,R); \BMA)$ the vector space of 
all $m^{\rm th}$ order tensors spanned by scaffolds defined on rooted diagram $(G, R)$ and 
edge weights in $\BMA$. 
\end{definition}

Observe that we may treat $G$ as an undirected graph when $\BMA$ is closed under the transpose map.

For the remainder of this section, we assume that edge weights are chosen from a 
\emph{coherent algebra} $\BMA$; that is, we 
assume $\BMA$ is a complex vector space of matrices with rows and columns indexed by $X$ 
\begin{itemize}
\item $A,B \in \BMA$ imply $AB \in \BMA$ {\sl (closure under ordinary multiplication)}
\item $A,B \in \BMA$ imply  $A\circ B \in \BMA$ {\sl (closure under entrywise multiplication)}
\item $A  \in \BMA$ implies $A^\top \in \BMA$ {\sl (closure under transpose)}
\item  $A  \in \BMA$ implies $\bar{A} \in \BMA$  {\sl (closure under entrywise conjugation)}
\item $I,J  \in \BMA$ {\sl ($\BMA$ is a ring with identity under both products)}
\end{itemize}

\medskip

\noindent {\bf Examples:} Assume $\BMA$ is a coherent algebra.
\begin{itemize}
\item For $G=K_2$ with $R=V(G)$, $\sbW( (G,R); \BMA)=\sbW($
\begin{tikzpicture}[black,node distance=0.5cm,
solidvert/.style={draw, circle,  fill=red, inner sep=2.2pt}]
\node[solidvert] (A1) at (0,0) {};
\node[solidvert] (A2) at (0.8,0) {};
   \path (A1)   [thick] edge node  [above] {} (A2);
\end{tikzpicture}
$; \BMA) = \BMA$;
\item For $|R|=2$ and $|E(G)| \le 4$, $\sbW( (G,R); \BMA) = \BMA$. This may be verified using our  basic observations about hollow nodes of degree one and two.
\end{itemize}

\begin{problem}
Given an association scheme $(X,\cR)$ with Bose-Mesner algebra $\BMA$ and an integer $m\ge 2$, observe
that
$$ \sbW( m ; \BMA) = \bigcup_{\substack{ (G,R) \\ |R| = m} } \sbW( (G,R) ; \BMA). $$
What is the smallest diagram $G$ with root nodes $R\subseteq V(G)$ satisfying $|R|=m$ such that
$\sbW( (G,R) ; \BMA) = \sbW( m ; \BMA)$?
\end{problem}

Let $G$ and $H$ be finite undirected graphs. An $H$-\emph{minor} in $G$ is a set $\{ G_v \mid v\in V(H)\}$ of
pairwise disjoint connected subgraphs of $G$ indexed by the nodes of $H$ such that there is an injection 
$\iota:E(H) \rightarrow E(G)$ that maps each edge $(x,y)$ of $H$ to some edge $(x',y')$ in $G$ with $x'\in G_x$ and
$y'\in G_y$.  Given $H$ with a specified ordered set $R=\{r_1,\ldots, r_k\}$ of distinct nodes in $V(H)$ and 
$G$ with a specified ordered set $R'=\{r'_1,\ldots, r'_k\}$ of nodes in $V(G)$, a \emph{rooted $H$-minor} 
with respect to $(H,R)$ and $(G,R')$ is an $H$-minor in $G$ satisfying $r'_i \in V(G_{r_i})$ for each $1\le i\le k$.

\begin{theorem}
\label{Thm:minors}
Assume $\BMA$ is a coherent algebra. If there is a rooted $H$-minor in $G$ with respect to $(H,R)$ and $(G,R')$,
then $\sbW((H,R); \BMA) \subseteq \sbW((G,R'); \BMA)$.
\end{theorem}

\begin{proof}
We apply Lemma \ref{Lem:seriespar}.
Each scaffold $\sS(H,R; w)$ is seen to be  equal to some scaffold $\sS(G,R'; w')$ by appropriately choosing
weights $w'(e)=I$ for any edge $e$ with both ends in the same subgraph $G_v$ and $w'(e)=J$ for edges outside
the image of map $\iota$ and ends in distinct subgraphs $G_v$ of the $H$-minor. $\Box$
\end{proof}

An immediate consequence of Theorem \ref{Thm:minors} is the following result.

\begin{corollary}
\label{Cor:tristarcontainment}
The vector spaces
\input{arxivdiagrams/inlineWDelta.tex} and  \input{arxivdiagrams/inlineWWye.tex} are both contained in the following two 
spaces of third order tensors spanned by scaffolds:
\begin{center}
\input{arxivdiagrams/WTriDelta.tex}  \qquad \input{arxivdiagrams/WK4.tex} 
\end{center}
\end{corollary}
\bigskip

If $\BMA$ is the Bose-Mesner algebra of an association scheme with standard bases $\{A_0,\ldots,A_d\}$ 
and $\{ E_0,\ldots, E_d\}$, intersection numbers $p_{ij}^k$, and Krein parameters $q_{ij}^k$, then one
may apply Lemma \ref{Lem:pijkqijk} to see that
\begin{center}
\input{arxivdiagrams/DeltaSpace.tex}
\end{center}
and
\begin{center}
\input{arxivdiagrams/WyeSpace.tex}
\end{center}

\noindent It is easy to see that \!\! $\displaystyle{
\input{arxivdiagrams/inlineWDelta.tex} }$ \!\!\!
is  invariant under the action  $(\cdot)$  and   that \!\!\! $\displaystyle{
\input{arxivdiagrams/inlineWWye.tex} }$ is  invariant under the action $\{\cdot\}$.

\begin{theorem}[Terwilliger~{\cite[Lemma 87]{terwnotes2009}}]
\label{Thm:DeltaWyebases}
For any symmetric association scheme, \\
{\bf (i)} the set 
\begin{center}
\input{arxivdiagrams/DeltaSpaceBasis.tex}
\end{center}
is an orthogonal basis for subspace $\displaystyle{
\input{arxivdiagrams/inlineWDelta.tex} }$ and\\
{\bf (ii)}  the set
\begin{center}
\input{arxivdiagrams/WyeSpaceBasis.tex}
\end{center}
is  an orthogonal basis for subspace $\displaystyle{
\input{arxivdiagrams/inlineWWye.tex} }$.
\end{theorem}

\begin{proof}
We compute scalar products as defined in Section \ref{Subsec:innerprod}:
\begin{center} 
\input{arxivdiagrams/TerOrthog1.tex} \\
\input{arxivdiagrams/TerOrthog1_1.tex} \\
$= \   |X| \delta_{i,r} \delta_{j,s} \delta_{k,t} v_k p_{ij}^k$.
\end{center}
Likewise, for the given generators of $\displaystyle{
\input{arxivdiagrams/inlineWWye.tex} }$, 
\begin{center} 
\input{arxivdiagrams/TerOrthog2.tex} \\
\input{arxivdiagrams/TerOrthog2_1.tex} 
\end{center}
using Lemma \ref{Lem:EiEjEksum}. It follows that the tensors given are pairwise orthogonal and nonzero.
$\Box$
\end{proof}

\begin{example}
Consider the association scheme of the Petersen graph, with Bose-Mesner algebra $\BMA$.  Terwilliger's results
tell us that $\dim \input{arxivdiagrams/inlineWDelta.tex} = 14$, the number of non-zero intersection numbers, and 
$\dim \input{arxivdiagrams/inlineWWye.tex} = 15$, the number of non-zero Krein parameters. Straightforward calculation
verifies that  
the automorphism group of the Petersen
graph has only 15 orbits on triples, so $\dim \sbW(3;\BMA) = 15$.
\end{example}

\begin{theorem}
Let $\BMA$ be a coherent algebra.  If either \input{arxivdiagrams/inlineWDelta.tex} $\subseteq$ \input{arxivdiagrams/inlineWWye.tex} or  \input{arxivdiagrams/inlineWWye.tex} $\subseteq$ \input{arxivdiagrams/inlineWDelta.tex}, then
\input{arxivdiagrams/inlineWK4minusedge.tex} $= \ \BMA$.
 \end{theorem}

\begin{proof}
Suppose we are given a scaffold 
\input{arxivdiagrams/K4e1.tex}.  If   \\
\begin{center}
\input{arxivdiagrams/DeltaLinCombWyes.tex}
\end{center}
then \\
\begin{center}
\input{arxivdiagrams/K4eLinComb.tex}
\end{center}

Likewise, if \hspace{0.5in}
\input{arxivdiagrams/WyeLinCombDeltas.tex} \hspace{0.5in}
then \\
\begin{center}
\input{arxivdiagrams/K4eLinComb2.tex} 
\end{center}
\end{proof}

\begin{problem}
Let $(X,\cR)$ be a symmetric association scheme. 
Determine necessary and sufficient conditions on $(X,\cR)$ for  $\displaystyle{
\input{arxivdiagrams/inlineWDelta.tex} = \input{arxivdiagrams/inlineWWye.tex} }$ to hold.
\end{problem}

These two spaces of third order scaffolds are fundamental. We will soon see their connection to the Terwilliger algebra. One can trace the origins of Terwilliger's subconstituent algebra to 1987 or earlier. In \cite{terPQ}, Terwilliger
explored these spaces of tensors in the context of $P$- and $Q$-polynomial association schemes.

The space $V^{\otimes 3}$, endowed with the inner product given in (\ref{Eqn:ip3}),  admits 
$$ \left( V^{\otimes 3} \right)_i 
= \spn \left\{ \hat{a} \otimes \hat{b} \otimes \hat{c} \mid a,b,c \in X, \ (a,b)\in R_i \right\} $$
as a $(\cdot)$-submodule and admits
$$ {\left( V^{\otimes 3} \right)^*}_j 
= \spn \left\{ \hat{a} \otimes \hat{b} \otimes E_j \hat{c} \mid a,b,c \in X \right\}. $$
as a $\{\cdot\}$-submodule.

The orthogonal projection $p_i : V^{\otimes 3} \rightarrow \left( V^{\otimes 3} \right)_i $ is given by 
$$ p_i( \hat{a} \otimes \hat{b} \otimes \hat{c} ) = \begin{cases} 
\hat{a} \otimes \hat{b} \otimes \hat{c} & \text{if} \ (a,b)\in R_i; \cr
\zero & \text{otherwise.}
\end{cases} $$
Dually, the orthogonal projection $p_j^* : V^{\otimes 3} \rightarrow  {\left( V^{\otimes 3} \right)^*}_j $ is given by 
$$ p_j^*( \hat{a} \otimes \hat{b} \otimes \hat{c} ) = 
\hat{a} \otimes \hat{b} \otimes E_j \hat{c} ~ .$$
Terwilliger introduces orthogonal bases $\{ \mathsf{e}_{st} \mid 0 \le s,t \le d\}$ and $\{ \mathsf{e}^*_{st} \mid 0 \le s,t \le d\}$ for an important subspace $\bW$ of $V^{\otimes 3}$, where
\begin{center}
\input{arxivdiagrams/est.tex} \hspace{0.5in} and  \hspace{0.5in} \input{arxivdiagrams/eststar.tex}
\end{center}
and also considers the subspace $\bW_1$ spanned by 
$\{ \mathsf{e}_{st} - \mathsf{e}_{ts} \mid 0 \le s,t \le d\}$ 
and by $\{ \mathsf{e}^*_{st} -\mathsf{e}^*_{ts}  \mid 0 \le s,t \le d\}$.

In his Lemma 2.12, Terwilliger proves that
\begin{center}
\input{arxivdiagrams/TerLem212pe.tex}  \hspace{0.5in} 
and  \hspace{0.5in} 
\input{arxivdiagrams/TerLem212pestar.tex}
\end{center}   
while
\begin{center}
\input{arxivdiagrams/TerLem212pstare.tex}
\hspace{0.5in} 
and
\hspace{0.5in} 
\input{arxivdiagrams/TerLem212pstarestar.tex}
\end{center} 
We recognize these as the edge action $\cE^{k,\ell}_{A_i}$ and node action $\cD^{i}_{E_j}$
introduced in Section \ref{Subsec:basics} and these appear in Jaeger's paper also.

Restricting to symmetric association schemes, we now revisit 
some of Terwilliger's orthogonality results. Recall our convention that, when taking inner products of tensors, corresponding components of the two diagrams are placed in corresponding positions spatially when
no confusion will arise.

\begin{center}
\input{arxivdiagrams/TerIPeijest.tex}
\end{center} 
as the sum of all entries
in $A_i A_j$ is $|X| v_i v_j$. Using the same approach, 
we check that the $\mathsf{e}_{st}^*$ tensors are pairwise orthogonal:

\begin{center}
\input{arxivdiagrams/TerIPeijeststar.tex}
\end{center}
since the matrix $E_i \circ E_j$ has trace $m_i m_j/|X|$.

\bigskip

\begin{lemma}[Terwilliger, Lemma~2.13]
$$ \langle p_i( \mathsf{e}_{jk}), p_i(\mathsf{e}_{st}) \rangle = 
\delta_{j,s} \delta_{k,t}  |X|  v_i p_{jk}^i~. $$
\end{lemma}

\noindent {\sl Proof:} 
\begin{center} 
\input{arxivdiagrams/TerProof1.tex} 
\end{center}
showing that the inner product is zero unless $j=s$ and $k=t$. Now assume $j=s$ and $k=t$ in which case
\begin{center} 
\input{arxivdiagrams/TerProof2.tex} 
\end{center}
and, since the sum of entries of $p_{jk}^i A_i$ is $|X| p_{jk}^i  v_i$, we are done. $\Box$

\bigskip

Likewise, the nonzero images among the tensors $p_j^*(\mathsf{e}_{st}^*) $ are pairwise orthogonal inside
the submodule ${\left( V^{\otimes 3} \right)^*}_j$.

\begin{lemma}[Terwilliger, Lemma~2.14]
$$ \langle p_j^*( \mathsf{e}_{hi}^*), p_j^*(\mathsf{e}_{st}^*) \rangle = \delta_{h,s} \cdot \delta_{i,t} \cdot |X| \cdot 
m_j \cdot q_{hi}^j~. \ \Box $$
\end{lemma}

But we get more interesting inner products if we pair the opposite way. The Gram matrices of the arrangements
in the next lemma play a key role in Terwilliger's results, including the balanced set condition.

\begin{lemma}[Terwilliger, Lemma~2.16]
$$ \langle p_j^*( \mathsf{e}_{hi}), p_j^*(\mathsf{e}_{st}) \rangle =  m_j \sum_{k=0}^d p_{hs}^k p_{it}^k P_{jk} $$
$$ \langle p_i( \mathsf{e}_{jk}^*), p_i(\mathsf{e}_{st}^*) \rangle = v_i \sum_{r=0}^d q_{js}^r q_{kt}^r Q_{ir}$$
\end{lemma}

\begin{proof}
To prove the first identity, we use basic operations on scaffolds as follows:
\begin{center} 
\input{arxivdiagrams/TerProof3.tex} \\
\input{arxivdiagrams/TerProof4.tex} \\
\input{arxivdiagrams/TerProof5.tex} 
\end{center}

Finally, we may use the orthogonality relation $Q_{kj} v_k = m_j P_{jk}$ to verify that the 
sum of all entries in this  final matrix is as given in the statement of the lemma.  The second statement
can be proven in a similar fashion. $\Box$
\end{proof}

In Appendix B, we collect identities for various inner products of basic third order scaffolds of this sort.

\subsection{Terwilliger Algebras}
\label{Subsec:Ter}

Let $(X,\cR)$ be an association scheme with Bose-Mesner algebra $\BMA$ having bases
$\{A_0,\ldots,A_d\}$ satisfying $A_i \circ A_j = \delta_{i,j} A_i $ and 
$\{E_0,\ldots,E_d\}$ satisfying $E_i E_j = \delta_{i,j} E_i $ as usual. Fix $x\in X$ and define 
$E_i^*(x)$ to be the diagonal  matrix with $(E_i^*(x))_{a,a} = (A_i)_{x,a}$; i.e., the $(a,b)$-entry 
of $E_i^*(x)$ is equal to one if $a=b$ with $(x,a)\in R_i$ and
equal to zero otherwise.  The {\em Terwilliger algebra} of $(X,\cR)$ with respect to base point $x$ is the matrix
algebra generated by the matrices $A_i$ and the matrices $E_i^*(x)$: 
$$ \Ter_x = \langle A_0, \ldots, A_d, \ E_0^*(x), \ldots, E_d^*(x) \rangle 
=  \langle E_0, \ldots, E_d, \ A_0^*(x), \ldots, A_d^*(x) \rangle$$
where 
$$ (A_j^*(x))_{a,b} = \begin{cases}
|X| (E_j)_{x,a} & \text{if} \ a=b; \cr
     0 & \text{otherwise.}
\end{cases} $$
Beginning with \cite{ter1}, an extensive theory of Terwilliger algebras, particularly for symmetric association schemes that are both metric and cometric, has developed over the past three decades. See \cite{terwnotes2009} for a relatively recent survey. Our goal here is simply to identify scaffolds encoding the matrices in such algebras.

Fix an association scheme and corresponding Bose-Mesner algebra  $\BMA$.
Let $\cT$ denote the vector space of all linear combinations of scaffolds of the form \\
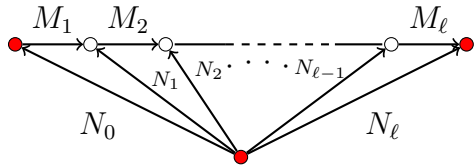
\begin{center}
\input{arxivdiagrams/terw_alg1.tex}
\end{center}
where $\ell \ge 1$ and $M_1,\ldots,M_\ell, N_0,\ldots, N_\ell \in \BMA$.  That is,
\begin{center}
$ \cT = \displaystyle{ \bigcup_{\ell = 1}^\infty}$ \ \  \input{arxivdiagrams/Wlfan.tex}  
\end{center}
The most basic third-order 
tensors of this form are the triangle and star  \\
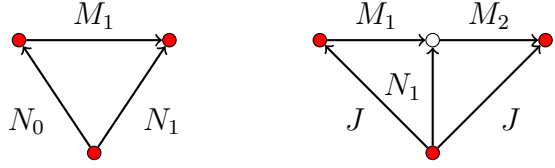
\begin{center}
\input{arxivdiagrams/terw_alg_tri_star.tex}
\end{center}
where $M_1,M_2,N_0,N_1 \in \BMA$ and $J$ is the all ones matrix. 

We define a product on elements of $\cT$ by gluing  diagrams as follows:\\
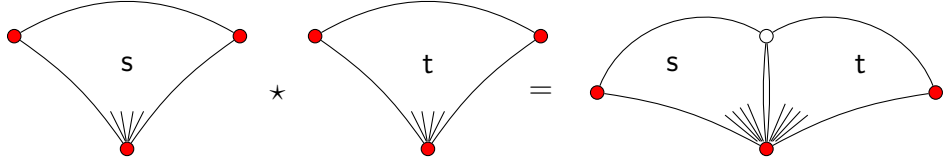
\begin{center}
\input{arxivdiagrams/terw_alg_product.tex}
\end{center}
Here, spacial arrangement is important and the root nodes at the bottom are identified while the 
rightmost root node of $\scafs$ is identified with the  leftmost root node of $\scaft$ and made hollow. 

Let us make this precise.

Given scaffolds 
$$ \scafs = \sS( G, \{b_0,a,b_\ell \};w) \qquad \mbox{and} \qquad 
\scaft = \sS(H,\{b'_0,a',b'_m \};w') $$
where 
$$ V(G)= \{ a, b_0, b_1,\ldots, b_\ell \}, \quad E(G) = \{ (b_0,b_1),\ldots, (b_{\ell-1},b_\ell), \ (a,b_0),\ldots,(a,b_\ell) \}$$
$$ V(H)= \{ a', b'_0, b'_1,\ldots, b'_m \}, \quad E(H) = \{ (b'_0,b'_1),\ldots, (b'_{m-1},b'_m), \ (a',b'_0),\ldots,(a',b'_m) \}$$
$$ w(b_{h-1},b_h) = M_h, \quad  w(a,b_h) = N_h, \qquad w'(b'_{h-1},b'_h) = M'_h, \quad  w'(a',b'_h) = N'_h,$$
we define 
$$ \scafs \star \scaft = \sS( K, \{b_0, a, b'_m \}; \hat{w}) $$
where
$$ V(K)= \{ a, b_0, b_1,\ldots, b_\ell ,b'_1,\ldots, b'_m \} $$
$$  E(K) = \{ (b_0,b_1),\ldots, (b_{\ell-1},b_\ell),  \  (b_\ell,b'_1),\ldots, (b'_{m-1},b'_m), \ (a,b_0),\ldots,(a,b_\ell), (a,b'_1),\ldots, (a,b'_m) \}$$
with edge weights $\hat{w}(a,b_\ell )= N_\ell \circ N'_0$ and 
$$ \hat{w}(b_{h-1},b_h) = M_h, \quad  \hat{w}(a,b_h) = N_h, \qquad \hat{w}(b_\ell,b'_1)=M'_1, \ \hat{w}(b'_{h-1},b'_h) = M'_h, \quad  \hat{w}(a,b'_h) = N'_h ~ .$$
The product is extended linearly to $\cT$.  

We state the following isomorphism without proof:

\begin{theorem}
\label{Thm:Tproduct}
The map $\xi : \cT \rightarrow \displaystyle{\bigoplus_{x\in X}} \  \Ter_x $ defined by 
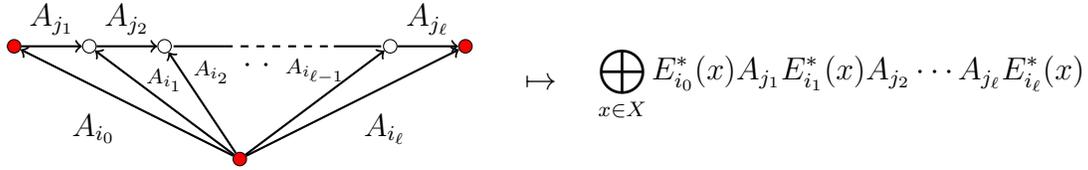
\begin{center}
\input{arxivdiagrams/terw_alg_basis.tex}
\end{center}
and extended linearly is a vector space isomorphism satisfying
$$ \xi( \scafs \star \scaft ) = \xi( \scafs ) \xi(  \scaft )  $$
where the product on the right is ordinary matrix product of block diagonal matrices.
\end{theorem}

This gives a natural interpretation of certain third-order scaffolds as elements of the direct sum of all Terwilliger algebras
$\Ter_x$ as $x$ ranges over the elements of $X$.   For example, since 
$$ E_i^*(x) A_j E_k^*(x) = \sum_{ \substack{ y,z\in X \\ (x,y)\in R_i, (y,z)\in R_j, (z,x)\in R_k}}  \hat{y} \otimes \hat{z}$$
we identify this matrix with $(E_i^*(x) A_j E_k^*(x)) \otimes \hat{x}$ and sum over $x\in X$ to obtain \input{arxivdiagrams/Talg1.tex}
Likewise, this isomorphism associates $\oplus_x A_i^*(x) E_j A_k^*(x)$, 
$\oplus_x A_i E_j^*(x) A_k$, and $\oplus_x  E_i A_j^*(x) E_k$, respectively, 
to the following scaffolds:

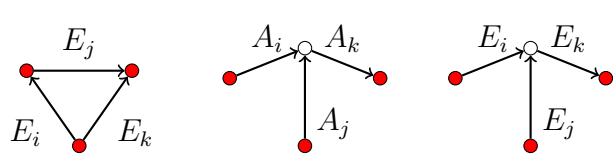
\begin{center}
\input{arxivdiagrams/Talg2.tex}
\end{center}

Paul Terwilliger [private communication] conjectures the following:
For a $Q$-polynomial bipartite distance-regular graph, the 
space of third order tensors of the form depicted on the left below is spanned by the subset of scaffolds 
with inner edges  all  having weight $A_t$, outer edges having weights $A_i$, $A_j$ and $A_k$. Further we 
obtain a basis  when we include only the  scaffolds of this sort where  $t+i+j +k \le d$.
\begin{center}
\input{arxivdiagrams/terwillconj.tex}
\end{center}

More importantly, Terwilliger conjectures that this space is both $(\cdot)$-invariant and $\{\cdot \}$-invariant; i.e., 
each of the maps $\cE^{k,\ell}_{A_i}$ and  $\cD^{i}_{E_j}$ map this space into itself. Terwilliger claims that, 
interpreted as in Theorem \ref{Thm:Tproduct}, this space is the full subconstituent algebra.

\begin{theorem}[Cf.~Terwilliger] \label{Thm:Talgcollapse}
Let $\BMA$ be the Bose-Mesner algebra of an association scheme.
\begin{itemize}
\item[(a)] The following are equivalent:
\begin{itemize}
\item[$\bullet$] 
\input{arxivdiagrams/inlineWtristar.tex} $=$ \input{arxivdiagrams/inlineWDelta.tex}
\item[$\bullet$]
\input{arxivdiagrams/inlineWK4.tex} $=$ \input{arxivdiagrams/inlineWDelta.tex}
\item[$\bullet$]  \qquad 
$\cT = $ \input{arxivdiagrams/inlineWDelta.tex}  \ .
\end{itemize}
\item[(b)] The following are equivalent:
\begin{itemize}
\item[$\bullet$] 
\input{arxivdiagrams/inlineWtristar.tex} $=$ \input{arxivdiagrams/inlineWWye.tex}
\item[$\bullet$] 
\input{arxivdiagrams/inlineWK4.tex} $=$ \input{arxivdiagrams/inlineWWye.tex}  
\item[$\bullet$]  \qquad 
$\cT = $ \input{arxivdiagrams/inlineWWye.tex}  \ .
\end{itemize}
\end{itemize}
\end{theorem}

\begin{proof}  By Theorem \ref{Thm:minors}, both \input{arxivdiagrams/inlineWDelta.tex}  and 
 \input{arxivdiagrams/inlineWWye.tex} are contained in both \input{arxivdiagrams/inlineWtristar.tex} and 
 \input{arxivdiagrams/inlineWK4.tex}. Moreover all of these spaces are contained in $\cT$.
 
 Now \ assume \ \input{arxivdiagrams/inlineWtristar.tex} $=$ \input{arxivdiagrams/inlineWDelta.tex}. \ \ We \  will \ 
 prove \ that \ \  $\cT = $\\ \input{arxivdiagrams/inlineWDelta.tex}. The remaining three parts of the proof are
 similar.

Given the scaffold \\
\input{arxivdiagrams/TequalDeltaproof1.tex} 

\bigskip

\noindent we apply standard rules ($\sSR{0}$, Lemma \ref{Lem:trivial}(ii) ) to manipulate $\scafs$ into a useful form:\\
\input{arxivdiagrams/TequalDeltaproof2.tex} 

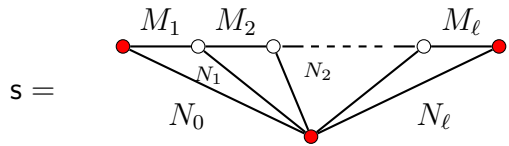
\bigskip \bigskip

\input{arxivdiagrams/TequalDeltaproof3.tex} 

By hypothesis, there exist matrices $R_1,S_1,T_1,\ldots$ in $\mathbb{A}$ satisfying

\input{arxivdiagrams/TequalDeltaproof0.tex}

\bigskip \bigskip

This substitution gives us \\
\input{arxivdiagrams/TequalDeltaproof4.tex} 

\bigskip \bigskip

\input{arxivdiagrams/TequalDeltaproof5.tex} 

\noindent and this reduction process can be repeated until we reach a linear combination of ``Deltas'', showing that 
$\mathsf{s} \ \in$  \input{arxivdiagrams/inlineWDelta.tex} .  $\Box$
\end{proof}

\subsection{The vector space of scaffolds of order two}
\label{Subsec:space2}

We consider the vector space of scaffolds of order two and various subspaces of this space. In particular, we wish 
to know, in the setting where edge weights belong to a Bose-Mesner algebra $\BMA$, when
such a subspace is no larger than $\BMA$ itself. Assume in this subsection that all edge weights belong to the 
Bose-Mesner algebra $\BMA$ of some  association scheme $(X,\cR)$.

First, since $\BMA$ is a nonzero subspace of $\Mat_X(\cx)$,  the vector space of scaffolds of order zero 
is simply $\cx$.   As we learned in Theorem  \ref{Thm:orbits}, the space of first order scaffolds has dimension 
equal to the number of orbits (on vertices) of the automorphism group of the scheme.

A \emph{circular planar graph} \cite{gitlerlopez} is an ordered pair $(G,R)$ where  $G$ is a graph embedded in the plane with a distinguished set $R \subseteq V(G)$ of nodes all appearing on the outer face. Let us say that a scaffold
$\sS(G,R;w)$ is \emph{planar} if $(G,R)$ is a circular planar graph.  For fixed $m$, the vector subspace of all  
$m^{\rm th}$ order planar scaffolds  with edge weights  in $\BMA$  is worthy of study. An 
$m$-\emph{terminal series-parallel graph} is a graph with a distinguished set of $m$ nodes which can be 
reduced via some sequence of series and parallel edge reductions to a graph on those $m$ nodes only.

\begin{example}
A circular planar graph  and its circular planar dual contain equally many terminal nodes. The following pair of 
examples illustrates the relationship between planar duality and duality in association schemes:
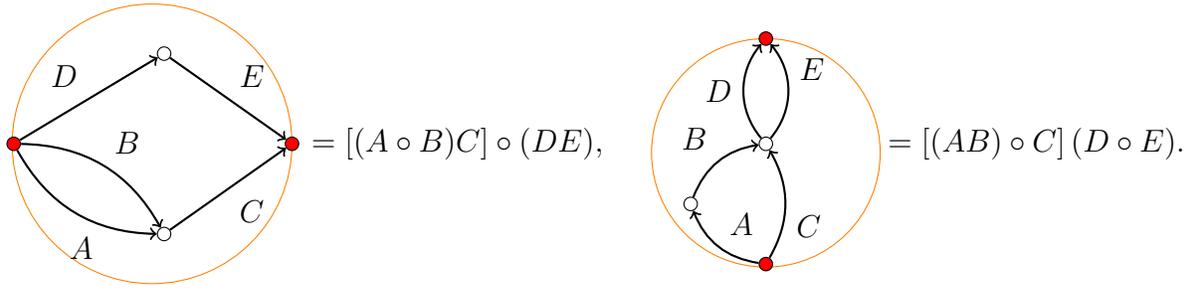
\begin{center}
\input{arxivdiagrams/dualpair.tex}
\end{center}
\end{example}

\begin{theorem}
Let $\mathcal{G}$ denote the set of all ordered pairs $(G,R)$ of two-terminal series parallel graphs with 
root nodes $R=\{r_1,r_2\}$.  For any coherent algebra $\BMA$, we have
$$ \BMA = \bigcup_{(G,R)\in \mathcal{G} } \sbW( (G,R); \BMA). $$
\end{theorem}

\begin{proof}
To prove forward containment is trivial: each $M \in \BMA$ is expressible as a scaffold whose underlying diagram 
$G$ is the complete graph on two nodes. 
If matrices $M,N \in \Mat_X(\cx)$ correspond to second order planar scaffolds $\scafs$ and $\scaft$,
respectively, then both their matrix product and their entrywise product are expressible as second order 
planar scaffolds as well.
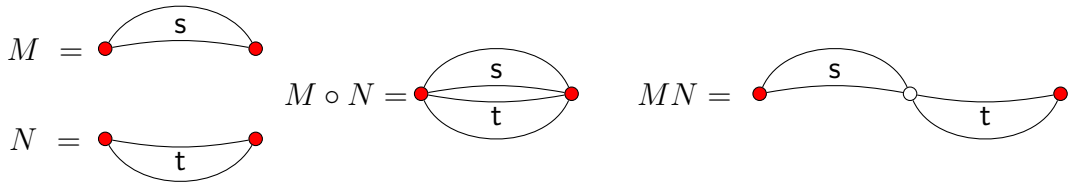
\begin{center}
\input{arxivdiagrams/planar.tex}
\end{center}
Applying these operations repeatedly, we see that any second order scaffold $\scafs$ having all edge weights in $\BMA$ whose underlying diagram can be constructed  from $K_2$ by successive subdivision and doubling of edges belongs to $\BMA$.  $\Box$
\end{proof}

It is well known (cf.~\cite[Exer.~32,p191]{diestel}) that a multigraph is series-parallel if and only if it contains no $K_4$ minor. It is easy to check that $K_{3,3}$ contains a $K_4$ minor. So it follows by Kuratowski's Theorem that 
every series-parallel graph is planar.  In general, the space of second order planar scaffolds  with edge 
weights in $\BMA$ can properly contain $\BMA$.  As an example, we consider the Doob graphs.

The Hamming graph $H(d,4)$ and the Doob graph {\sf Doob}$(s,t)$ with $2s+t=d$ have the same parameters \cite[Sec.~9.2B]{bcn}. The first scaffold below takes on the same value for both, but the second scaffold 
gives different values for different values of $(s,t)$:

\begin{center}
\input{arxivdiagrams/Doob1.tex}
\end{center}

Indeed, the first scaffold computes the number of labelled triangles in a graph $\Gamma$ with adjacency matrix
$A_1$. Since $A_1^2 = 3dA_0+2A_1+2A_2$ for both the Hamming and Doob graph,  this
is the sum of entries of $A_1^2 \circ A_1 = 2A_1$.  But a Doob graph is a Cartesian product of $s$ copies
of $K_4$-free Shrikhande graphs with $t$ copies of $K_4$, and only these latter factors contribute to the 
count of  labelled 4-cliques.

In fact,
\begin{center}
\input{arxivdiagrams/Doob2.tex}
\end{center}
where $B$ is the adjacency matrix of the subgraph $H$ of {\sf Doob}$(s,t)$ whose edges are
those belonging to 4-cliques;  up to isomorphism, this graph $H$ consists of $16^s$ copies of the Hamming
graph $H(t,4)$.

\begin{example}
If $A_1$ is the adjacency matrix of the Shrikhande graph $\Gamma$ on sixteen vertices, one may easily
devise a second order scaffold which is the adjacency matrix of a graph $\Sigma$, isomorphic to $\Gamma$,
which appears as a subgraph of the complement $\Gamma'$:
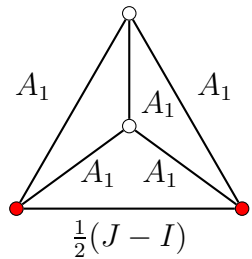
\begin{center}
\input{arxivdiagrams/Shrik.tex}
\end{center}
This planar second order scaffold does not belong to the Bose-Mesner algebra of $\Gamma$.
\end{example}

In the case of the Hamming graph, every second order planar scaffold corresponds to some matrix in the
Bose-Mesner algebra. In Theorem \ref{Thm:planar2terminal} below, we will prove  that, whenever
the association  scheme is both triply regular and dually triply regular, 
all second order planar scaffolds fall within the Bose-Mesner algebra. 
.

Following Hestenes and Higman \cite{hesthig}, a strongly regular graph $\Gamma$ is said to enjoy
the $t$-\emph{vertex condition} if, for any graph $G$ on at most $t$ nodes and any two distinguished nodes
$a,b\in V(G)$ the number of graph homomorphisms from $G$ to $\Gamma$ mapping $a$ to $x$ and $b$ to
$y$ depends only on whether $x$ and $y$ are equal, adjacent, or non-adjacent. A recent investigation on
this topic is Reichard's paper \cite{reichard}.

Inspired by this, we say an association scheme $(X,\cR)$ with Bose-Mesner algebra $\BMA$
enjoys the $t$-\emph{vertex condition} if every second order scaffold $\sS(G,R;w)$ with $t$ or fewer nodes
and edge weights in $\BMA$ belongs to $\BMA$.  We prove in Section \ref{Sec:triplyreg} below that every
triply regular association scheme satisfies the 4-vertex condition.

%
%
%
\section{Duality of planar scaffolds}
 \label{Sec:planar}
 
 The obvious duality between the scaffold identities
 \begin{center}
 \input{arxivdiagrams/pijkqijk.tex} \\
\end{center}
 extends to dual pairs of theorems in some cases, as we've seen.  We claim that these are instances of a much more general phenomenon.
 
 A \emph{scaffold equation} of order $r$ is an equation of the form
 $$ \sum_{k=1}^m \alpha_k \sS_k = \zero $$
 where each $\alpha_k$ is a scalar, each $\sS_k$ is a scaffold of order $r$, $\zero$ is the zero tensor of order $r$,
 and a bijection $\zeta_{j,k}$ is specified (or understood) between the root nodes of $\sS_k$ and $\sS_j$  for each 
 $j$ and $k$ in a consistent manner; i.e., we assume $\zeta_{i,j} \circ \zeta_{j,k}=\zeta_{i,k}$ for each $i,j,k$ and $\zeta_{k,k}$ is the identity map. We note that, throughout this paper, this correspondence of tensor components
 has been conveniently indicated  pictorially by consistent spatial placement of the root nodes.  Note that, for fixed
 $d$, the $P$-polynomial condition and the $Q$-polynomial condition can both be encoded as finite systems of 
 scaffold equations.

Some circular planar graphs admit multiple, inequivalent, embeddings in a disk. 
We may define an augmented graph $G^+$ by adding an additional node $\infty$ whose neighbors are exactly those nodes $v\in R$; it is immediate that $(G,R)$ is a circular planar graph if and only if $G^+$ is a planar graph. 
Moreover, by a theorem of Tutte, if $G^+$ if 3-edge-connected, then this planar embedding is unique.

Let $\scafs = \sS(G,R;w)$ be a planar symmetric scaffold with a fixed embedding in a closed disk where all root nodes 
appear on the boundary. Assume, for simplicity, that $w(e) \in \{A_0,\ldots,A_d\}$ for each edge $e$.
In order to define the \emph{dual scaffold} $\scafs^\dagger$, we first construct a dual graph
$G^\dagger$ which has one node for each face of this embedding. (This circular planar dual can be obtained from the
planar dual of graph $G^+$ by deleting the dual edges corresponding to edges of $G^+$ incident
to the node $\infty$. So, in contrast to the planar dual of $G$, this
graph has $m=|R|$ nodes on the infinite face, which has been subdivided by the $m$ segments of the boundary 
of the disc.)  Each directed edge $e$ of $G$ is rotated $90^\circ$ counterclockwise  to give an edge 
$e^\dagger$ of $G^\dagger$ joining the two faces it bounds.  The distinguished (``root'') nodes of the dual 
scaffold are those $m$ faces incident to the bounding disk.  The edge weights are then given by $w(e^\dagger)=E_j$
where $w(e)=A_j$.  This map is extended linearly as in the notion of a duality map.

\begin{conjecture}
Suppose we have a collection $\{ \sS_k \}$ of $m^{\rm th}$ order  planar scaffolds where all edge weights belong to 
the set of symbols $A_0,A_1,\ldots, A_\delta$. Assume that, for all  association schemes 
with $d\ge \delta$ classes,  the scaffold equations $\sum_{k=1}^{n_j} \alpha_{jk} \sS_k = \zero$ 
($1\le j\le n$) together imply
the scaffold equation $\sum_{k=1}^{n_0} \beta_k \sS_k = \zero$.  

Then, for any  association scheme with $d\ge \delta$ classes, the dual scaffold equations
$$ \sum_{k=1}^{n_j} \alpha_{jk} \sS_k^\dagger = \zero  \qquad (1\le j\le n)$$
together imply the  dual scaffold equation
 $$\sum_{k=1}^{n_0} \beta_k \sS_k^\dagger = \zero.$$
\end{conjecture}

This conjecture allows us to map identities to identities. 
As an example, we now give three obviously equal scaffolds for $P$-polynomial schemes and the  
dual scaffolds which are equal for all $Q$-polynomial schemes. (Equality is easily shown using the Isthmus Lemma.)

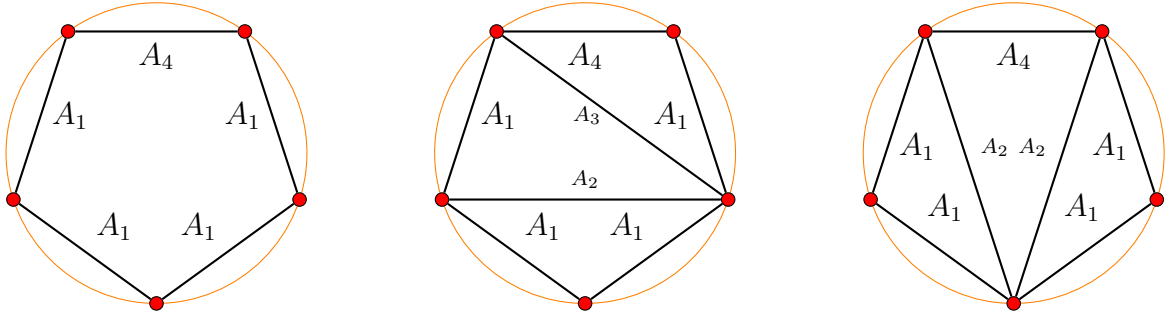
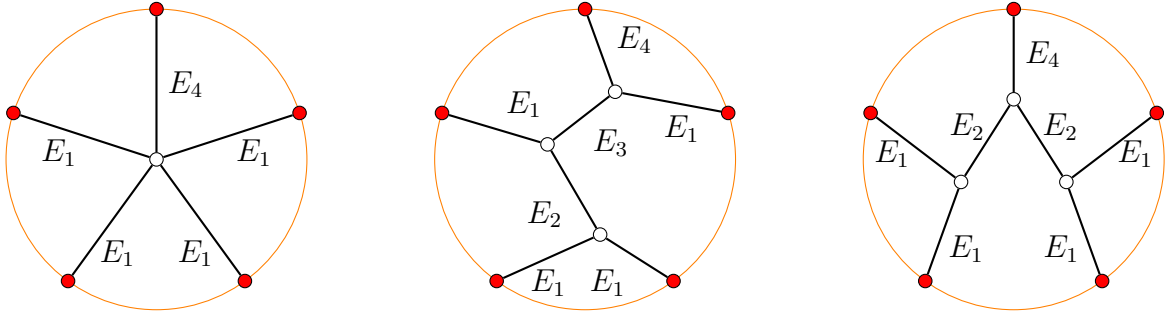
\input{arxivdiagrams/pentagon0.tex} \hspace{0.5in}
\input{arxivdiagrams/pentagon1.tex} \hspace{0.5in}
\input{arxivdiagrams/pentagon2.tex} 

\vspace{1in}

\input{arxivdiagrams/fivestar0.tex} \hspace{0.5in}
\input{arxivdiagrams/fivestar1.tex} \hspace{0.5in}
\input{arxivdiagrams/fivestar2.tex} 

%
%
%
\section{Triply regular association schemes}
\label{Sec:triplyreg}

We call a $d$-class symmetric association scheme \emph{triply regular} if, for all $0\le i,j,k,r,s,t \le d$,
the following identity of tensors holds for some scalar $\tau_{i,j,k}^{r,s,t}$: \\
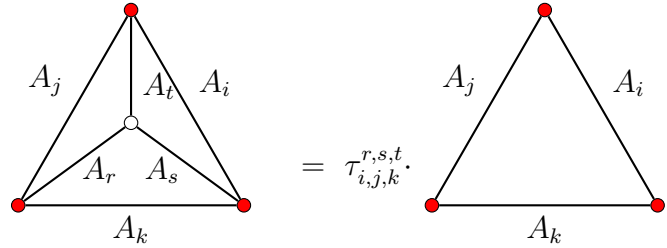
\begin{center}
\input{arxivdiagrams/triplyregular.tex}
\end{center}

In other words, for all $x,y,z \in X$ and all indices $r,s,t$ the number of vertices $w$ which are $r$-related to $x$, $s$-related to $y$ and $t$-related to $z$ just depends on $r,s,t$ and those indices $i,j,k$ for which
 $(x,y)\in R_i$, $(y,z)\in R_j$ and $(z,x)\in R_k$ and not on the choice of vertices $x,y,z$ themselves.

\bigskip

Dually, let's call a symmetric association scheme \emph{dually triply regular} if the following identity of tensors holds for some scalar $\sigma_{i,j,k}^{r,s,t}$: 

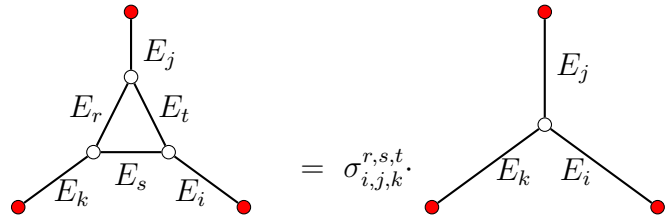
\begin{center}
\input{arxivdiagrams/dualtriplyregular.tex}
\end{center}

In other words, for all $x,y,z\in X$ and all indices $i,j,k,r,s,t$
$$ \sum_{u,v,w\in X} (E_i)_{x,u} (E_j)_{y,v} (E_k)_{z,w}  \  (E_r)_{v,w} (E_s)_{w,u} (E_t)_{u,v} $$
is a scalar multiple of 
$$ \sum_{w' \in X} (E_i)_{x,w'} (E_j)_{y,w'} (E_k)_{z,w'} $$
independent of the choice of $x,y,z$.

\begin{theorem}
\label{Thm:tripreg1}
Let $(X,\cR)$ be a symmetric association scheme with Bose-Mesner algebra $\BMA$.  Then
\begin{itemize}
\item[(i)]  $(X,\cR)$ is  triply regular   if and only if 
\input{arxivdiagrams/inlineWDelta.tex} \!\!\!\! $=$ \input{arxivdiagrams/inlineWK4.tex} \!\!\! ;
\item[(ii)] $(X,\cR)$ is  dually triply regular  if and only if 
\input{arxivdiagrams/inlineWWye.tex} \!\!\!\! $=$ \input{arxivdiagrams/inlineWtristar.tex} \!\!\! .
\end{itemize}
\end{theorem}

\begin{proof}
Clearly \input{arxivdiagrams/inlineWDelta.tex} is a subspace of \input{arxivdiagrams/inlineWK4.tex}. If the triply regular
condition holds, then we obviously have containment in the other direction as well. Conversely, observe that
\begin{center}
\input{arxivdiagrams/IPK4Delta.tex}
\end{center}
and that the ``Delta'' scaffolds with edge weights in $\{A_0,\ldots,A_d\}$ are pairwise orthogonal by Theorem
\ref{Thm:DeltaWyebases}. This implies that, if \input{arxivdiagrams/tripregK4.tex} belongs to 
\input{arxivdiagrams/inlineWDelta.tex}, then it must be  a scalar multiple of  \input{arxivdiagrams/tripregK3.tex}.  
The second claim is proved in a similar manner.   $\Box$
\end{proof}

We now have another way to interpret Theorem \ref{Thm:Talgcollapse}: the Terwilliger algebra $\cT$ of $(X,\cR)$
satisfies $\cT = $ \input{arxivdiagrams/inlineWDelta.tex} if and only if  $(X,\cR)$ is triply regular and satisfies
 $\cT = $ \input{arxivdiagrams/inlineWWye.tex} if and only if  $(X,\cR)$ is dually triply regular. Moreover, we have

\begin{theorem}
\label{Thm:tripreg2}
Let $(X,\cR)$ be a symmetric association scheme with Bose-Mesner algebra $\BMA$.  Then
\begin{itemize}
\item[(i)]  $(X,\cR)$ is  triply regular   if and only if 
\input{arxivdiagrams/inlineWDelta.tex} \!\!\!\! $=$ \input{arxivdiagrams/inlineWtristar.tex} \!\!\! ;
\item[(ii)] $(X,\cR)$ is  dually triply regular  if and only if 
\input{arxivdiagrams/inlineWWye.tex} \!\!\!\! $=$ \input{arxivdiagrams/inlineWK4.tex} \!\!\! .
\end{itemize}
\end{theorem}

\begin{proof}
In both cases, forward containment is given by Corollary \ref{Cor:tristarcontainment}. Next assume 
that $\BMA$ is the Bose-Mesner algebra
is triply regular and consider the following series of expansions using $\sSR{0'}$ and $J=\sum_\ell A_\ell$: 

\begin{center}
\input{arxivdiagrams/TriplyRegularTriStar.tex}
\end{center}
To prove reverse containment for part {\sl (ii)}, we work with the duals of these diagrams. Applying Rule $\sSR{0}$
and the expansion $I = \sum_\ell E_\ell$, we compute
\begin{center}
\input{arxivdiagrams/DualTriplyRegularK4.tex}
\end{center}
\end{proof}

\begin{proposition}
\label{Prop:triply_regular_4vertex}
Every triply regular symmetric association scheme satisfies the 4-vertex condition.
\end{proposition}

\begin{proof}
Let $(X,\cR)$ be a triply regular association scheme with Bose-Mesner algebra $\BMA$.  
We must prove that, for every choice of edge weights $M_1,\ldots,M_6 \in \BMA$, the scaffold
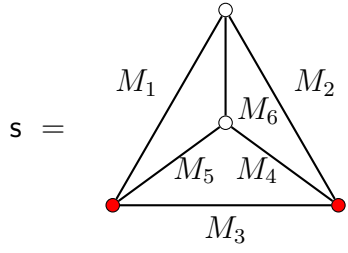
\begin{center}
\input{arxivdiagrams/tripreg4proof1.tex}
\end{center}
belongs to $\BMA$. By linearity, we may assume each $M_i \in \{A_0,\ldots,A_d\}$ and employ
the triply regular property to write 
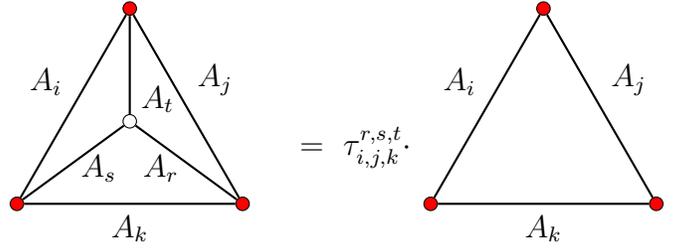
\begin{center}
\input{arxivdiagrams/tripreg4proof2.tex}
\end{center}
Now the fundamental scaffold at right is simply the sum of elementary tensors $\hat{x} \otimes \hat{y} \otimes \hat{z}$
over all ordered triples $(x,y,z) \in X^3$ with $(x,y) \in R_k$, $(x,z)\in R_i$, $(z,y)\in R_j$. Summing over $z\in X$,
we find
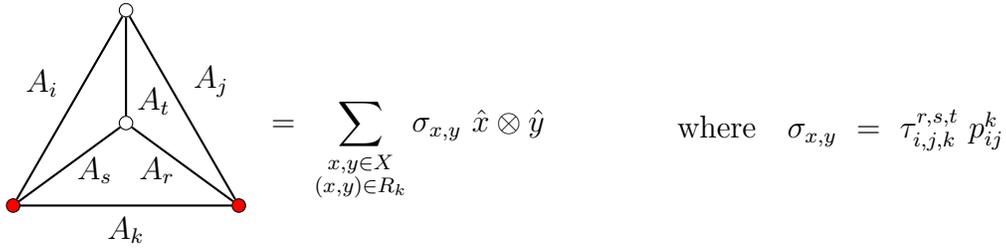
\begin{center}
\input{arxivdiagrams/tripreg4proof3.tex}
\end{center}
so that the tensor at left is a scalar multiple of $A_k$.
Replacing various $M_i$ in $\scafs$ by matrices $I$ and $J$, as needed, one obtains the result for any second order scaffold on at most four
nodes. $\Box$
\end{proof}

As Jaeger \cite[Prop.~5]{jaeger} points out,  Epifanov's Theorem  (see \cite{truemper})
establishes that every connected undirected plane graph can be reduced to the trivial graph with one node and no edge via some finite sequence of $\Delta$--Y and
Y--$\Delta$ transformations, together with extended series-parallel reductions. 

\begin{theorem}
\label{Thm:planar2terminal}
Let $(X,\cR)$ be an association scheme with Bose-Mesner algebra $\BMA$. If $(X,\cR)$ is both triply
regular and dually triply regular, then the vector space  spanned by all second order planar scaffolds with 
edge weights in $\BMA$ is equal to $\BMA$.
\end{theorem}

\begin{proof}
Epifanov's Theorem tells us that any two-terminal planar graph is reducible,  via a sequence of 
series-parallel reductions, $\Delta$--Y and Y--$\Delta$ transformations, to a single edge joining those two
terminals. The algebra $\BMA$ is identified with scaffolds defined on 
these single-edge two terminal diagrams with edge weights in
$\BMA$. By Lemma \ref{Lem:seriespar}, neither a series nor a parallel reduction changes the space of edge
weights provided this space is closed under both ordinary and entrywise multiplication. Lemma \ref{Lem:trivial}
allows us to first replace any Y configuration by a $K_4$ configuration using edge weight $J$ on any edges introduced.
Appropriate linear expansions allow us to express any such scaffold as a linear combination of scaffolds whose local 
configuration is a $K_4$ with all edge weights in $\{A_0,\ldots,A_d\}$. The triply regular property then allows us to replace this $K_4$ subdiagram with a $\Delta$ configuration.

The reverse transformation is a bit more delicate. Given a diagram with a $\Delta$ configuration, we first
employ node splitting, as necessary, using Lemma \ref{Lem:trivial} to obtain a diagram where each node 
of the triangle has degree three in the overall diagram. We then express this scaffold as a linear combination
of scaffolds where the six edge weights all lie in $\{ E_0,\ldots, E_d\}$, then applying the dually triply regular 
property to replace this $\Delta$ with a Y configuration. Up to three series reductions are then needed to
exactly mimic the $\Delta$--Y transformation of graph theory. 
By Epifanov's Theorem, after a finite number of steps, each 
scaffold in this linear combination is reduced to a linear combination of second order scaffolds each with a single
edge. As a result, we have expressed our original scaffold as a linear combination of elements of the Bose-Mesner 
algebra. $\Box$
\end{proof}

\medskip

\noindent {\bf Note:} Note that this does not imply that all triply regular association
schemes are Schurian. Theorem \ref{Thm:orbits} requires the entire space of second order scaffolds
to have dimension $d+1$ and this theorem only considers the space spanned by planar scaffolds.

\section*{Acknowledgments}

Various elements of this paper have been presented in talks and were included in earlier drafts which benefited
from informed critiques. 
The author thanks Rosemary Bailey, Sylvia Hobart, Gavin King, Jack Koolen, Xiaoye Liang, Bojan Mohar, 
Eric Moorhouse, Akihiro Munemasa, 
Dan Perreault,  Georgina Quinn, Hajime Tanaka, Paul Terwilliger, Andrew Uzzell, Jason Williford --- and surely others --- for helpful comments. Paul, in particular, provided valuable advice on key parts of the paper. I am grateful to Pi 
Fisher for his help with TIKZ diagrams.
This work was supported, in part, through a grant from the National Science Foundation (DMS Award \#1808376) which is gratefully acknowledged.

%
%
\section*{Appendix}

\appendix

\section{Rules for scaffold manipulation}
\label{App:rules}

In this appendix, we summarize the rules for manipulation of scaffolds. In the case of symmetric association
schemes, all references to directed edges may be replaced by equivalent language referring to edges. The
rules here are given informally with reference to their precise statement in the body of the paper.

\begin{itemize}
\item[$\sSR{0}$] (split node rule) Lemma \ref{Lem:trivial}(ii): We may split a node, solid or hollow, introducing a 
new hollow vertex and choosing $I$ as the new edge weight.

\input{arxivdiagrams/splitnode.tex}
\item[$\sSR{0'}$] (superfluous edge rule) Lemma \ref{Lem:trivial}(i): Between any two nodes of our diagram, we may
insert a new edge $e$ with $w(e)=J$. Conversely, edges with weight $J$ may be deleted.

\input{arxivdiagrams/insertJ1.tex}  \hspace{.4in}
or  \hspace{.4in}
\input{arxivdiagrams/insertJ2.tex}  \hspace{.2in} . 
\item[$\sSR{1}$] (series reduction) Lemma \ref{Lem:seriespar}(i): We may suppress a hollow node of degree two
by taking the matrix product of the two edge weights.

\input{arxivdiagrams/Ruleseries.tex}

\item[$\sSR{1'}$] (parallel reduction) Lemma \ref{Lem:seriespar}(ii):  We may replace two parallel edges by a single
edge by taking the entrywise product of the two edge weights.

\input{arxivdiagrams/Ruleparallel.tex}
\end{itemize}

\bigskip
 
\noindent {\bf Note:} Scaffold manipulation rules  $\sSR{2}$ through $\sSR{4'}$ apply within the scope of Bose-Mesner algebras; edge weights follow standard notational conventions for association schemes.

\begin{itemize}
\item[$\sSR{2}$] (vanishing intersection number) Lemma \ref{Lem:pijkqijk}: Any scaffold containing a directed triangle $a,b,c$ with 
$w(a,b)=A_i$, $w(b,c)=A_j$, $w(a,c)=A_k$ where $p_{ij}^k =0 $ is the zero tensor.

\input{arxivdiagrams/vanishpijk.tex}

\item[$\sSR{2'}$] (vanishing Krein parameter)  Lemma \ref{Lem:pijkqijk}: Any scaffold containing a hollow node $x$ of degree three
with neighbors $a,b,c$ such that 
$w(a,x)=E_i$, $w(b,x)=E_j$, $w(x,c)=E_k$ where $q_{ij}^k =0 $ is the zero tensor.

\input{arxivdiagrams/vanishqijk.tex}

\item[$\sSR{3}$] (pinched star)   Equation (\ref{Eq:drumstick}):

\input{arxivdiagrams/drumstickidentity.tex} 

\medskip

and this is the zero tensor if $q_{ij}^k=0$.

\item[$\sSR{3'}$] (hollow triangle) Equation (\ref{Eq:hollowtriangle}):  

\input{arxivdiagrams/Rulewing.tex}

\medskip

and this is the zero tensor if $p_{ij}^k=0$.
\item[$\sSR{4}$] (Isthmus) Lemma \ref{Lem:Isthmus}: If $q_{jk}^e \cdot q_{\ell m}^e =0$ for all 
$e\neq h$, then

\input{arxivdiagrams/SuzukiIsthmus1.tex}

\noindent and

\input{arxivdiagrams/SuzukiIsthmus2.tex} 

\item[$\sSR{4'}$] (Dual isthmus)  Lemma \ref{Lem:DualIsthmus}: If $p_{hi}^e \cdot p_{jk}^e =0$ for 
all $e\neq \ell$, then

\input{arxivdiagrams/DualIsthmus1.tex}

\noindent and
 
\input{arxivdiagrams/DualIsthmus2.tex}

\item[$\sSR{5}$] (substitution) Proposition \ref{Prop:glue}: If  $\scaft_1$ and $\scaft_2$ are scaffolds 
on the same set $R$ of roots such  that $\scaft_1 = \scaft_2$ and $R' =\{r_1,\ldots,r_\ell\} \subseteq R$, 
then for any scaffold $\scafs$ and any  root nodes $u_1,\ldots,u_\ell$ in the rooted diagram of $\scafs$, 
we have $\scafs +_\xi \scaft_1 = \scafs +_\xi \scaft_2$ where $\xi(u_i) = r_i$.

\input{arxivdiagrams/RuleSubstitution.tex}

\item[$\sSR{6}$] (multilinearity): If scaffolds $\scafs$ and $\scafs_1,\ldots,\scafs_n$ are identical except in
their weight on one edge $e$ where $w(e)=M$ in $\scafs$ and $w(e)=N_\ell$ in $\scafs_\ell$ ($1\le \ell \le n$)
where $M = \sum_{\ell=1}^n \alpha_\ell N_\ell$, then 
$\scafs = \sum_{\ell=1}^n \alpha_\ell \scafs_\ell $.

\input{arxivdiagrams/linearity.tex}

\item[$\sSR{7}$] (Transpose property): Reversing the direction of an edge in diagram $G$ is equivalent to
replacing the weight of that edge by its transpose.

\input{arxivdiagrams/RuleTranspose.tex} 

\item[$\sSR{8}$] (Commutative property):  If $a \in V(G)$ is a hollow node incident to just two edges $e$ and
$e'$ where $w(e) w(e') = w(e')w(e)$, then swapping the weights on these edges leaves the scaffold unchanged.

\input{arxivdiagrams/Rulecommute.tex}

\item[$\sSR{9}$] (degree one vertices) Lemma \ref{Lem:degreeonehollow}: Assuming constant row sum
or column sum (as appropriate) on the edge weight, a hollow node of degree one may be deleted.

\input{arxivdiagrams/degreeone.tex}

\item[$\sSR{10}$] (order reduction) Lemma \ref{Lem:orderred}: Equality is preserving in passing from root
nodes $R$ to a proper subset $R'$.

\input{arxivdiagrams/RuleOrderReduction.tex}

\item[$\sSR{11}$] (bilinear maps) If $\sS$, $\scafs$ and $\scaft$ are scaffolds where $\scafs=\scaft$ and the 
products are appropriately defined, then, since the operation in Equation (\ref{Eqn:bilin1}) is well-defined, 
$$ \langle \langle \sS , \scafs \rangle \rangle \ = \  \langle \langle \sS , \scaft \rangle \rangle ~ .$$
\end{itemize}

\section{Inner products of common third order scaffolds}
\label{App:ip}

Here we use (\ref{Eqn:ip3}) to compute some inner products of third order scaffolds 
attached to (commutative) association schemes  without proof.  Note that $\bar{E}_j = E^\top_j$ and that, 
in the symmetric case, edge orientations may be ignored. Only in special cases 
is the result expressible in terms of association scheme parameters.

\bigskip

\input{arxivdiagrams/IPtriAA.tex}

\vspace{-0.3in}

\input{arxivdiagrams/IPtriEA.tex}

\vspace{-0.3in}

\input{arxivdiagrams/IPtriEE.tex}

\input{arxivdiagrams/IPstarEE.tex}

\input{arxivdiagrams/IPstarEA.tex}

\input{arxivdiagrams/IPstarAA.tex}

\input{arxivdiagrams/IPstartriAA.tex}

\input{arxivdiagrams/IPstartriEA.tex}

\input{arxivdiagrams/IPstartriEE.tex}

\input{arxivdiagrams/IPtristarEstarE.tex}


\input{arxivdiagrams/IPtristarEstarA.tex}


\input{arxivdiagrams/IPtristarAA.tex}


\input{arxivdiagrams/IPtristarEA.tex}


\input{arxivdiagrams/IPtristarEE.tex}


\begin{thebibliography}{XX}
\bibitem{banito} 
E.~Bannai and T.~Ito. 
\underline{Algebraic Combinatorics I: Association Schemes}. 
Ben\-ja\-min-Cummings, Menlo Park, 1984.

\bibitem{bcn}
A.~E.~Brouwer, A.~M.~Cohen and A.~Neumaier.
\underline{Distance-Regular Graphs}. Springer-Verlag, Berlin, 1989.

\bibitem{cgs} 
P.~J.~Cameron, J.-M.~Goethals, and J.~J.~Seidel.
The Krein condition, spherical designs, Norton algebras and permutation
groups. \textit{Proc.~Kon.~Nederl.~Akad.~Wetensch. (Indag. Math.)} {\bf 40}
no. 2 (1978), 196--206.

\bibitem{srgsrg}
P.~J.~Cameron, J.-M.~Goethals, and J.~J.~Seidel.
Strongly regular graphs with strongly regular subconstituents.
\textit{J.~Algebra} {\bf 55} (1978), 257--280.

\bibitem{cooljur}
K.~Coolsaet and A.~Juri\v{s}i\'{c}. 
Using equality in the Krein conditions to prove nonexistence of certain distance-regular graphs.
\textit{J.~Combin.~Theory Ser.~A}, \textbf{115} no.~6 (2008),  1086--1095.

\bibitem{circularplanar}
E.~B.~Curtis, D.~Ingerman, and J.~A.~Morrow.
Circular planar graphs and resistor networks.
\textit{Linear Algebra Appl.} \textbf{283} (1998), 115--150.

\bibitem{DRGsurvey}
E.~R.~van Dam, J.~H.~Koolen, and H.~Tanaka.
Distance-Regular Graphs.
\textit{Electronic J.~Combin.}  Dynamic Survey  DS22, 2016

\bibitem{del} 
P.~Delsarte. 
An algebraic approach to the association schemes of coding theory.  
{\it Philips Res.\ Reports Suppl.} {\bf 10} (1973).

\bibitem{dickie}
G.~A.~Dickie.
``$Q$-polynomial  structures for association schemes and distance-regular graphs'', 
Ph.D.~thesis, University of Wisconsin, 1995.

\bibitem{diestel}
R.~Diestel,
\underline{Graph Theory}.
Springer-Verlag, Heidelberg, 1997 (Electronic Edition, 2005)

\bibitem{gavrsudavida}
A.~Gavrilyuk, S.~Suda and J.~Vidali.
On tight 4-designs in Hamming association schemes.
\textit{Preprint}, July 2019. \url{http://front.math.ucdavis.edu/1809.07553}

\bibitem{gitlerlopez}
I.~Gitler and I.~L\'{o}pez.
On topological spin models and generalized $\Delta-Y$ transformations.
\textit{Adv.~Appl.~Math.} {\bf 32} (2004), 263--292.

\bibitem{godsil}
C.~D.~Godsil.
\underline{Algebraic Combinatorics}.
Chapman and Hall, New York, 1993.

\bibitem{hesthig}
M.~D.~Hestenes and D.~G.~Higman.
Rank 3 groups and strongly regular graphs. pp.~141--160 in:
\underline{Proc.~Symp.~Applied Math.} Amer.~Math.~Soc., Providence RI, 1971.

\bibitem{jaeger}
F.~Jaeger.
On spin models, triply regular association schemes, and duality. 
{\it J.~Algebraic Combin.} {\bf 4} (1995), 103--144.

\bibitem{lovasz}
L.~Lov\'{a}sz.
\underline{Large Networks and Graph Limits}.
AMS Colloquium Publications \#60, 
Amer.~Math.~Soc., 2012.

\bibitem{symsub}
W.~J.~Martin.
Symmetric designs, sets with two  intersection numbers and Krein parameters of incidence graphs.
\textit{J.~Combin.~Math.~Combin.~Comput.} {\bf 38} (2001), 185--196.

\bibitem{mtsurvey}
W.~J.~Martin and H.~Tanaka.
Commutative association schemes.
\textit{Europ. J. Combin.} \textbf{30} no. 6 (2009), 1497--1525.

\bibitem{mw}
W.~J.~Martin and J.~S.~Williford.
There are finitely many $Q$-polynomial association schemes with given first multiplicity at least three. 
\textit{European J.~Combin.} {\bf 30} (2009), 698--704. 

\bibitem{MesnerBhatt}
D.~M.~Mesner and P.~Bhattacharya.
Association schemes on triples and a ternary algebra.
\textit{J.~Combin.~Theory Ser.~A}, \textbf{55} (1990),  204--234.

\bibitem{neupen1}
S.~Penjic and A.~Neumaier.
A unified view of inequalities for distance-regular graphs. Part I. 
Manuscript (2018). (\url{https://www.mat.univie.ac.at/~neum/ms/uniDRG1.pdf})

\bibitem{neupen2}
S.~Penjic and A.~Neumaier.
A unified view of inequalities for distance-regular graphs. Part II. 
Manuscript (2018). (\url{https://www.mat.univie.ac.at/~neum/ms/uniDRG2.pdf})

\bibitem{pech}
Ch.\ Pech.
On highly regular strongly regular graphs.
\textit{Preprint} (2016), arXiv:1404.7716.

\bibitem{reichard}
S.~Reichard.
Strongly regular graphs with the $7$-vertex condition.
\textit{J. Algebraic Combin.}
{\bf 41} (2015), 817--842.
  
\bibitem{suzimprim}
H. Suzuki,
Imprimitive $Q$-polynomial association schemes.
\textit{J. Algebraic Combin.}
{\bf 7} (1998), 165--180.

\bibitem{suztwoq}
H. Suzuki,
Association schemes with multiple $Q$-polynomial structures.
\textit{J. Algebraic Combin.}
{\bf 7} (1998), 181--196.

\bibitem{terPQ}
P.~Terwilliger.
A characterization of $P$-~and $Q$-polynomial association schemes.
{\it J.~Combin.~Theory, Ser.~A} {\bf 45} (1987), 8--26.

\bibitem{ter1}
P.~Terwilliger.
The subconstituent algebra of an association scheme, (part I).
{\it J.~Algebraic Combin.} {\bf 1} (1992), 363--388.

\bibitem{terwnotes2009}
P.~Terwilliger.
Course lecture notes, Math 846 Algebraic Graph Theory, Spring term 2009, University of Wisconsin
\url{http://www.math.wisc.edu/~terwilli/Htmlfiles/part2.pdf}

\bibitem{truemper}
K. Truemper.
On the delta-wye reduction for planar graphs.
{\it J.~Graph Theory} {\bf 13} no.~2 (1989), 141--148.

\end{thebibliography}
\end{document}